\documentclass[12 pt]{article}
\usepackage{amsmath}
\usepackage{amsfonts}
\usepackage{amssymb}
\usepackage{amsthm}
\usepackage{thmtools}
\usepackage{hyperref}
\usepackage{mathtools}
\usepackage{enumitem}
\usepackage{subcaption}
\usepackage{graphicx}

\newtheorem{thm}{Theorem}[section]
\newtheorem{lem}[thm]{Lemma}

\newtheorem{cor}[thm]{Corollary}
\newtheorem{defn}[thm]{Definition}
\newtheorem{ques}[thm]{Question}
\newtheorem{rmk}[thm]{Remark}

\title{Certain connect sums of torus knots with infinitely many non-characterizing slopes}
\author{Konstantinos Varvarezos}
\begin{document}
\maketitle

\begin{abstract}
For a knot $K,$ a slope $r$ is said to be characterizing if for no other knot $J$ does $r$-framed surgery along $J$ yield the same manifold as $r$-framed surgery on $K.$ Applying a condition of Baker and Motegi, we show that the knots $T_{2,2n+3}\#T_{-2,2n+1}$ have infinitely many non-characterizing slopes.
\end{abstract}


\section{Introduction}

Given a knot $K$ in $S^3$ and an $r \in \mathbf{Q} \cup \{\infty\},$ we denote the \textit{Dehn surgery} on $K$ with slope $r$ by $S^3_r(K).$  Given a knot $K,$ a slopes $r$ is called \emph{characterizing} for $K$ if $S^3_r(K)\cong S^3_{r}(J)$ implies $K=J.$ 

It is known that all (nontrivial) slopes are characterizing for the unknot \cite{KMOS} as well as for the  the trefoil and the figure-eight knot \cite{OSzChar}. On the other hand, there are several knots known to have non-characterizing slopes; for instance the knots $6_2,$ $6_3$ as well as several other 7- and 8-crossing knots in the tables have been shown to admit infinitely many non-characterizing slopes \cite{AT}. In the case of composite knots, Osoinach showed that 0 is a non-characterizing slope for the knot $4_1 \# 4_1,$ the connect sum of the figure-eight knot with itself \cite{Oso}.

In this work, we consider the family of knots $K_n = T_{2,2n+3}\#T_{-2,2n+1},$ which are connected sums of torus knots.

\begin{figure}
\centering
\includegraphics[width=.4\textwidth]{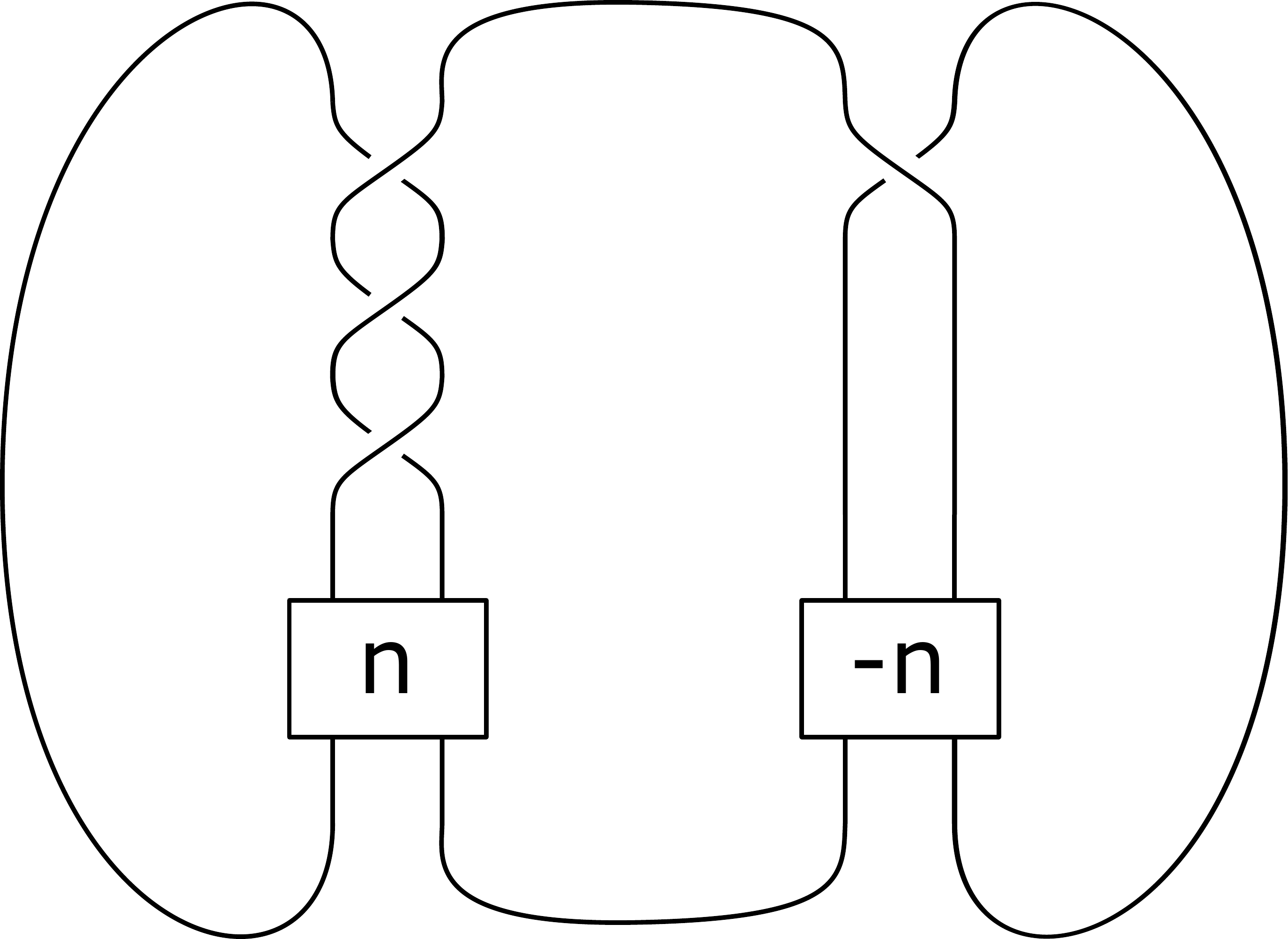}
\caption{The knot $K_n=T_{2,2n+3}\#T_{-2,2n+1}.$ Here the boxes represent $n$ full right-handed or left-handed twists of the strands passing through, depending on whether the number in the box is positive or negative, respectively.}
\label{fig:tcon}
\end{figure}

We show that each knot in this family has infinitely many non-characterizing slopes:

\begin{restatable}{thm}{main}\label{thm:main}
For each positive integer $n,$ the knot $K_n=T_{2,2n+3}\#T_{-2,2n+1}$ has infinitely many non-characterizing slopes.
\end{restatable}

To the author's knowledge, $K_1 = T_{2,5}\#T_{-2,3}$ is the smallest (in terms of crossing number) composite knot known to have infinitely many non-characterizing slopes. Indeed, it is still unknown whether composite knots with fewer crossings (connect sums involving the trefoils and/or the figure-eight knot) have any non-characterizing slopes.

\begin{ques}
Do any of the knots $3_1 \# 3_1,$ $3_1 \# \overline{3}_1,$ or $3_1 \# 4_1$ have non-characterizing slopes?
\end{ques}

The construction used to show the main result for $K_n$ has a natural extension to connect sums of pairs of 2-bridge knots which differ by a single crossing change. Thus, one might ask:

\begin{ques}
Let $K$ and $J$ be 2-bridge knots differing by a single crossing change. Does $K \# \overline{J}$ have infinitely many non-characterizing slopes?
\end{ques}

\section{Preliminaries}

In this section, we review the relevant definitions and results that will be used in the demonstration of our main result.
Throughout this work, we use the following notational conventions. If $M$ and $M'$ are oriented 3-manifolds, we write $M\cong M'$ to denote that there is an orientation-preserving homeomorphism from $M$ to $M'.$ If $K$ and $J$ are knots in $S^3,$ we write $K=J$ to mean $K$ is isotopic to $J.$ Moreover, we denote the mirror of $K$ by $\overline{K}.$ 

If $p/q \in \mathbf{Q}\cup\{\infty\}$ is a (reduced) fraction (here $\infty = 1/0$), we denote the Dehn surgery on $K$ along the slope $p/q$ by $S^3_K(p/q).$

\subsection{Characterizing slopes}

For a given knot, a rational slope is said to be \emph{characterizing} if surgery on the knot along that slope characterizes the knot. More precisely:

\begin{defn}
Given a knot $K\subset S^3,$ a slope $r\in \mathbf{Q}$ is called a \emph{characterizing slope} for $K$ if for any knot $K'\subset S^3,$
\[
S^3_K(r)\cong S^3_{K'}(r) \Rightarrow K=K'.
\]
\end{defn}

Lackenby has shown that every knot has infinitely many characterizing slopes \cite{Lack}. In the case of hyperbolic knots, McCoy has improved Lackenby's result by providing more effective bounds; in particular all but finitely many slopes $p/q$ with $|q|\geq 3$ will be characterizing \cite{McC2}.

For the simplest knots, it turns out that all rational slopes are characterizing. In particular, Kronheimer, Mrowka, Ozsv\'ath and Szab\'o showed this is true for the unknot \cite{KMOS}, and Ozsv\'ath and Szab\'o proved it for the trefoil and figure-eight knot \cite{OSzChar}.

On the other hand, plenty of knots are known to have slopes which are not characterizing. Osoinach gave a general method, called annulus twisting, of producing pairs of knots (in fact, an infinite family of knots) with homeomorphic 0-surgeries \cite{Oso}. As a consequence of this method, it was shown that 0 is not a characterizing slope for the knot $4_1\#4_1.$ Teragaito applied this method to find infinitely many knots sharing $4$-surgery with the knot $9_{42},$ hence $4$ is non-characterizing for this knot \cite{Ter}. It was later shown by Abe, Jong, Luecke, and Osoinach that every integer $n$ is a non-characterizing slope for infinitely many knots (in fact, there is an infinite family of knots all with homeomorphic $n$-surgeries) \cite{AJLO}.  Using a different construction, Piccirillo, in the process of showing the Conway knot is not slice, found that 0 is not a characterizing slope for that knot \cite{Pic}. That method, which makes use of what is called an RGB link, has been shown to be related to the annulus twist and to the methods used in this work \cite{Tag}.

Baker and Motegi have given a criterion for showing a knot has \emph{infinitely many} non-characterizing slopes, which they used to show the knots $9_{42}, $ $P(-3,3,5), $ and $8_6$ have this property \cite{BM}. Moreover, they showed that every knot is the companion of some prime satellite with infinitely many non-characterizing slopes, and similarly every knot is the summand of a composite knot with this property. Abe and Tagami applied these methods to show that the knots $6_1$ and $6_2$ (as well as several other low-crossing knots) have infinitely many non-characterizing slopes \cite{AT}. These are the lowest crossing-number knots known to have infinitely many non-characterizing slopes. 

Recently, Baldwin and Sivek have shown that the knot $5_2$ has at least one non-characterizing slope, while also showing that the same knot has many characterizing slopes \cite{BS}. In particular, every rational slope except the positive integers was shown to be characterizing for this knot. Work of Ni and Zhang has shown that the knot $5_1=T_{2,5}$ also has a preponderance of characterizing slopes; indeed, all slopes are characterizing except possibly the negative integers and the slopes $0,\pm \frac{1}{2},\pm\frac{1}{3},1$ \cite{NZ,NZ2}. McCoy has shown similar kinds of results for torus knots more generally as well as for the hyperbolic knot $P(-2,3,7)$ \cite{McC,McC3}.

\subsection{A theorem of Baker and Motegi}

Here we present the main result that we shall use to show the existence of infinitely many non-characterizing slopes, which is the following theorem of Baker and Motegi:

\begin{thm}[Theorem 1.3 of \cite{BM}]\label{thm:BM}
Suppose $K$ and $c$ are the components of a two-component link in $S^3$ with $c$ an unknot such that:
\begin{itemize}
\item $(0,0)$-framed surgery on the link $K \cup c$ yields $S^3,$ and
\item $c$ is not isotopic to a meridian of $K.$
\end{itemize}
Then $K$ has infinitely many non-characterizing slopes.
\end{thm}

The proof of this theorem actually shows that such a knot $K$ has infinitely many \emph{integral} non-characterizing slopes. Moreover, it shows that all but possibly finitely many integer slopes are non-characterizing, although \emph{a priori} the particular integers may depend on the knot $K.$

\subsection{Banded annulus presentations}

Following \cite{AJOT,AT2,AT,Tag}, we describe the construction of a banded annulus presentation for a knot.
Let $A\subset S^3$ be an embedded annulus.
 A \emph{band} is an embedding $b:[0,1]\times [0,1] \hookrightarrow S^3.$ We say that a band $b$ is \emph{properly attached} to $A$ if:
 \begin{itemize}
 \item  $b\left(\{0,1\}\times [0,1]\right) = \partial A \cup \mathrm{Im}(b)$
 \item $\mathrm{Im}(b) \cap \mathrm{Int}( A)$ consists only of ribbon singularities, and
 \item $A\cup \mathrm{Im}(b)$ is an immersed surface\footnote{Some authors require this surface to be orientable, but we shall not require this; see Remark \ref{rmk:ori} below.}
 \end{itemize}

See Figure \ref{fig:banded_annulus_ex} for an example illustration.

A knot $K$ is said to admit a \emph{banded annulus presentation} if $K$ is the boundary of $A\cup \mathrm{Im}(b)$ for some embedded annulus $A\subset S^3$ and band $b$ properly attached to $A.$ Further, a \emph{special} banded annulus presentation is one such that the annulus $A$ is a Hopf band.

\begin{figure}
\centering
\begin{subfigure}{.45\textwidth}
\centering
\includegraphics[scale=.3]{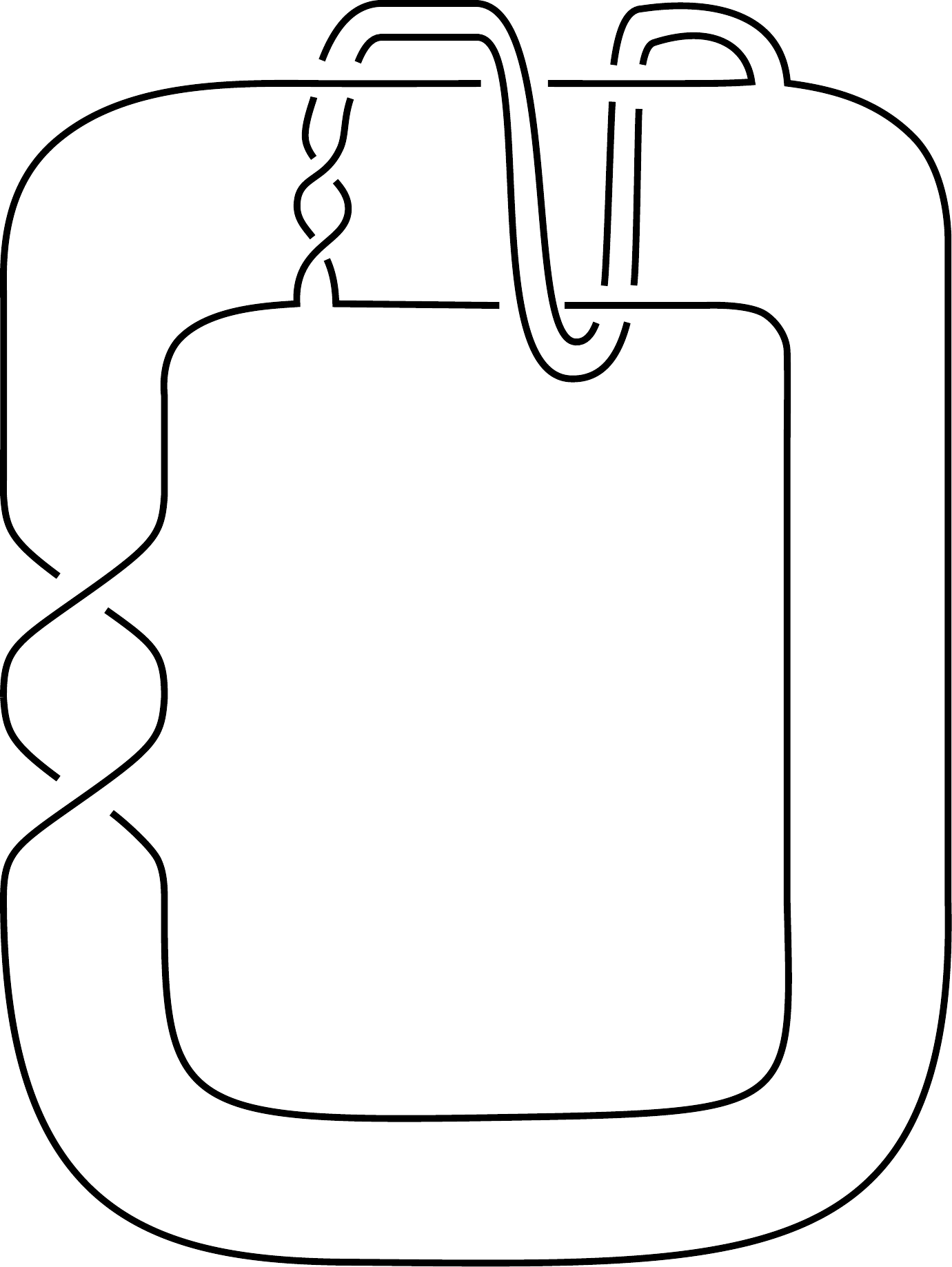}
\caption{the knot $T_{2,5}\#T_{-2,3}$}
\end{subfigure}
\begin{subfigure}{.45\textwidth}
\centering
\includegraphics[scale=.3]{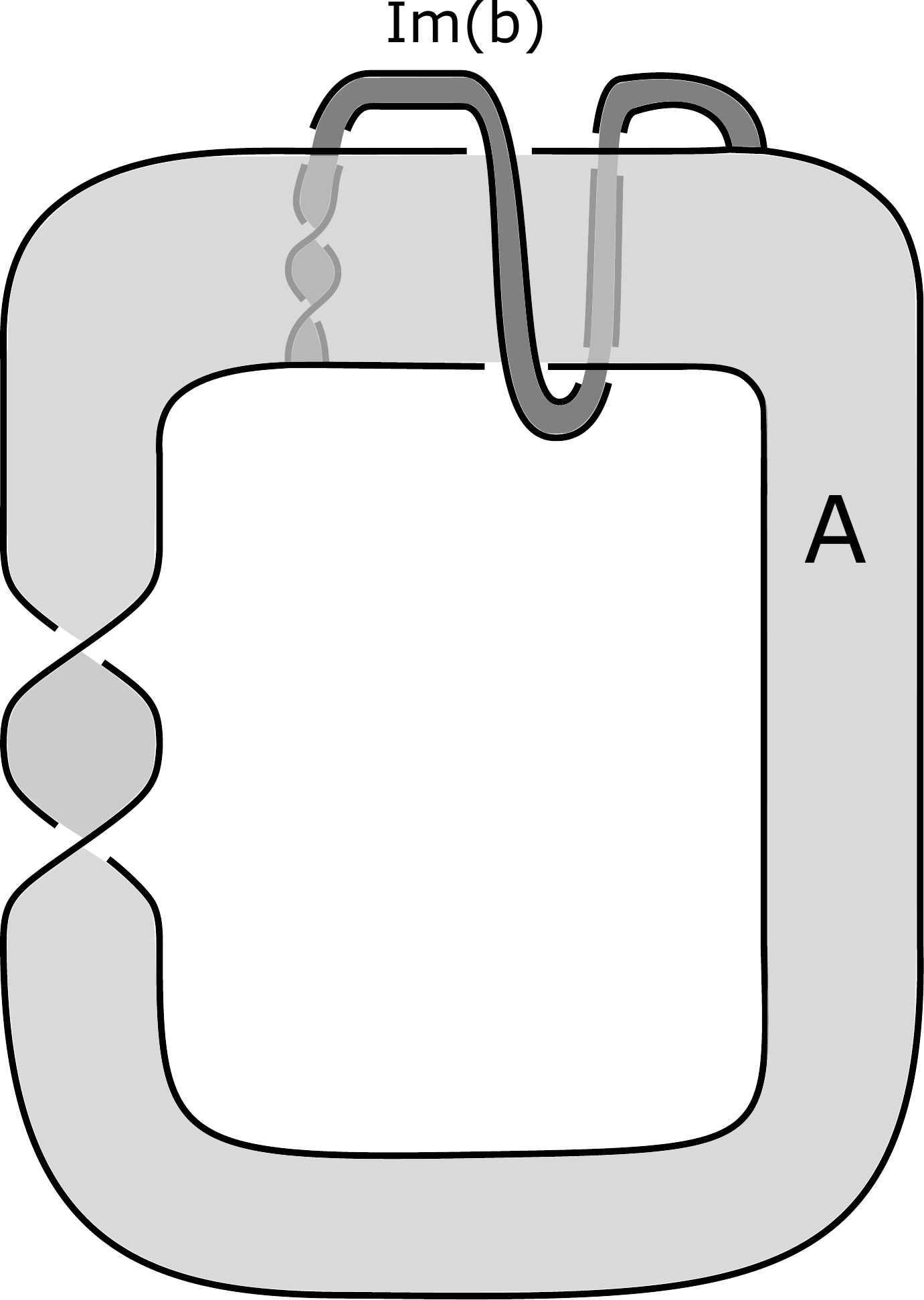}
\caption{annulus with band properly attached}
\end{subfigure}
\caption{The knot $T_{2,5}\#T_{-2,3}$ (left) expressed as the boundary of a banded annulus (right)}
\label{fig:banded_annulus_ex}
\end{figure}

Knots arising as the boundaries of a special banded annului will be of interest to us because such knots admit in general interesting ``companion" unknots. More precisely, if $K$ admits a special banded annulus presentation, there is an unknot $\beta_K \subset S^3 \setminus K$ which will satisfy (at least) half of the Baker-Motegi criterion (this construction can also be interpreted in terms of the so-called dualizable patterns developed by Gompf and Miyazaki \cite{GM} and used by Miller and Piccirillo \cite{MP}):
\begin{lem}
[Lemma 4.4 of \cite{AT}; see also Section 5 of \cite{MP}]\label{lem:00}
Let $K\subset S^3$ be a knot that admits a special banded annulus presentation, and let $\beta_K$ be an unknot in the exterior of $K$ as in Figure \ref{fig:banded_annulus_sch}. Then $(0,0)$-surgery on the link $K\cup \beta_K$ yields $S^3.$
\end{lem}
\begin{rmk}\label{rmk:ori}
In \cite{AT} and \cite{MP}, a banded annulus is assumed to yield an \emph{orientable} surface. However, orientability is not actually necessary for this particular result. Indeed, that $(0,0)$-surgery on $K_n\cup \beta_{K_n}$ yields $S^3$ can be seen directly in our case from an application of Kirby calculus to Figure \ref{fig:link_gen}.
\end{rmk}
\begin{figure}
\centering
\includegraphics[scale=.3]{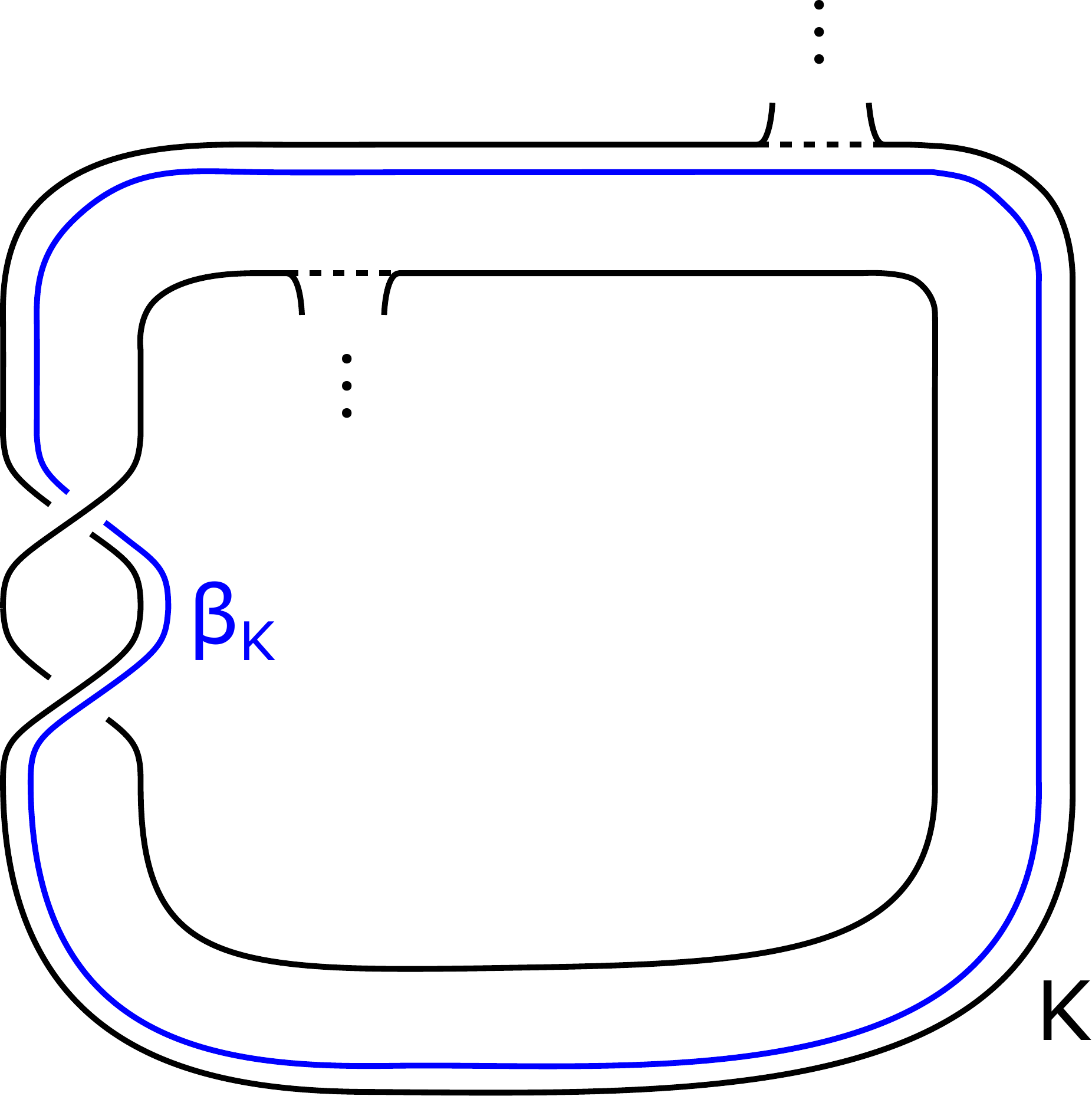}
\caption{Schematic showing the ``companion" unknot $\beta_K$ corresponding to a knot $K$ admitting a special banded annulus presentation. The dashed regions represent the places where the band attaches to the boundary of the annulus.}
\label{fig:banded_annulus_sch}
\end{figure}

We claim that each of the knots $K_n = T_{2,2n+3}\#T_{-2,2n+1}$ admits a special banded annulus presentation. This can be seen in Figure \ref{fig:seq}, where performing a band surgery along the blue band yields $K_n.$

\begin{figure}
\centering
\includegraphics[width=\textwidth]{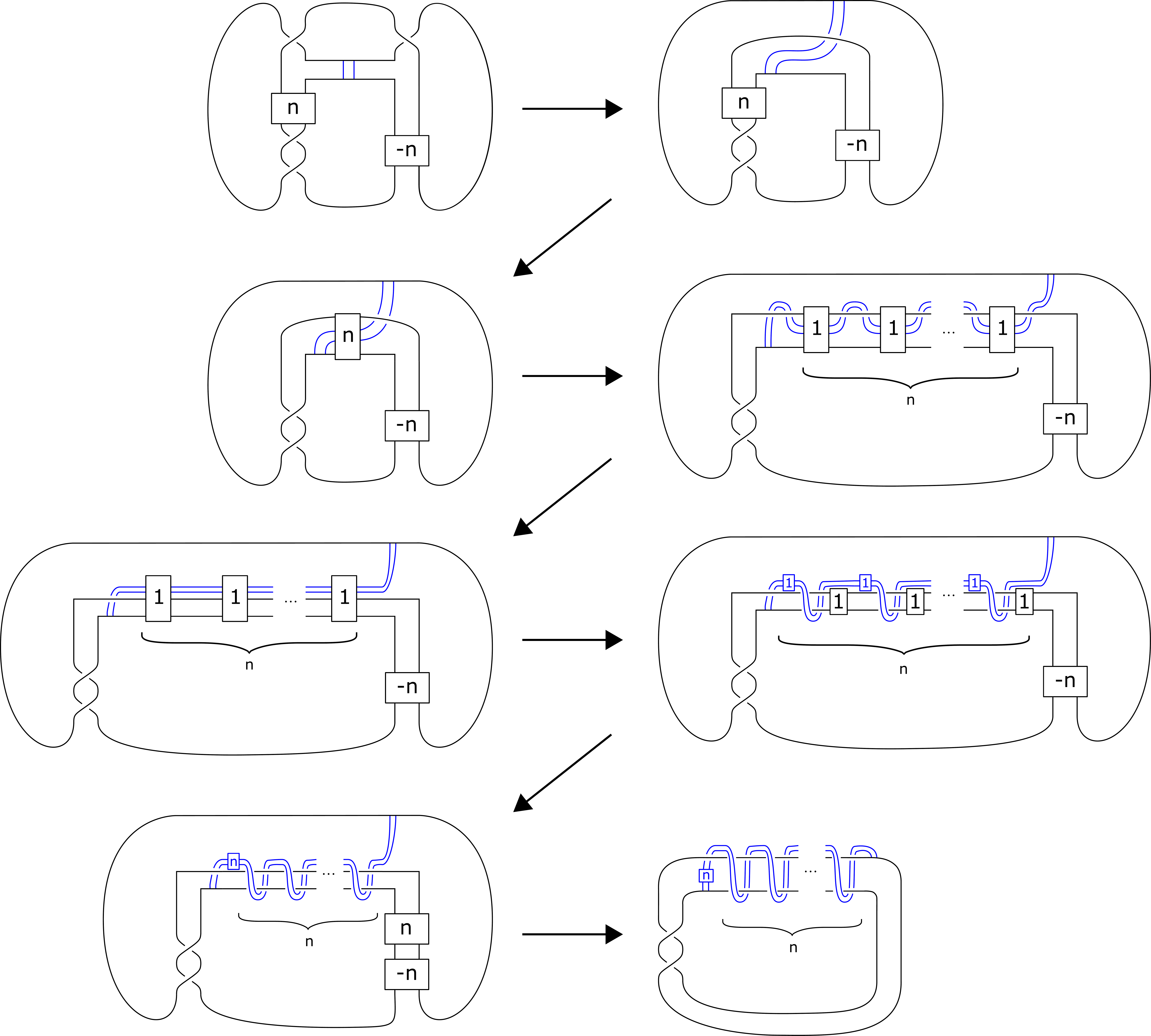}
\caption{A sequence of ``moves" demonstrating that $K_n$ admits a special banded annulus presentation.}
\label{fig:seq}
\end{figure}

We thus have:
\begin{lem}\label{lem:banded}
For each positive integer $n,$ the knot $K_n = T_{2,2n+3}\#T_{-2,2n+1}$ admits a special banded annulus presentation.
\end{lem}

\subsection{The Conway polynomial}

Here we recall some useful definitions and facts about the Conway polynomial for an oriented link $L\subset S^3.$ The \emph{Conway polynomial} of $L,$ denoted $\nabla_L,$ is characterized by the following equation:
\[
\nabla_L(t^{-1/2}-t^{1/2}) = \Delta_L(t)
\]
where $\Delta_L(t)$ is the symmetric, normalized Alexander polynomial of $L.$ It is known that the Conway polynomial satisfies the following skein relation. If oriented links $L_+,$ $L_-,$ and $L_0$ differ locally as in Figure \ref{fig:skein}, then:
\begin{equation}\label{eq:conway_skein}
\nabla_{L_+}(z)-\nabla_{L_-}(z)=z\nabla_{L_0}(z).
\end{equation}

\begin{figure}
\centering
\includegraphics[width=.5\textwidth]{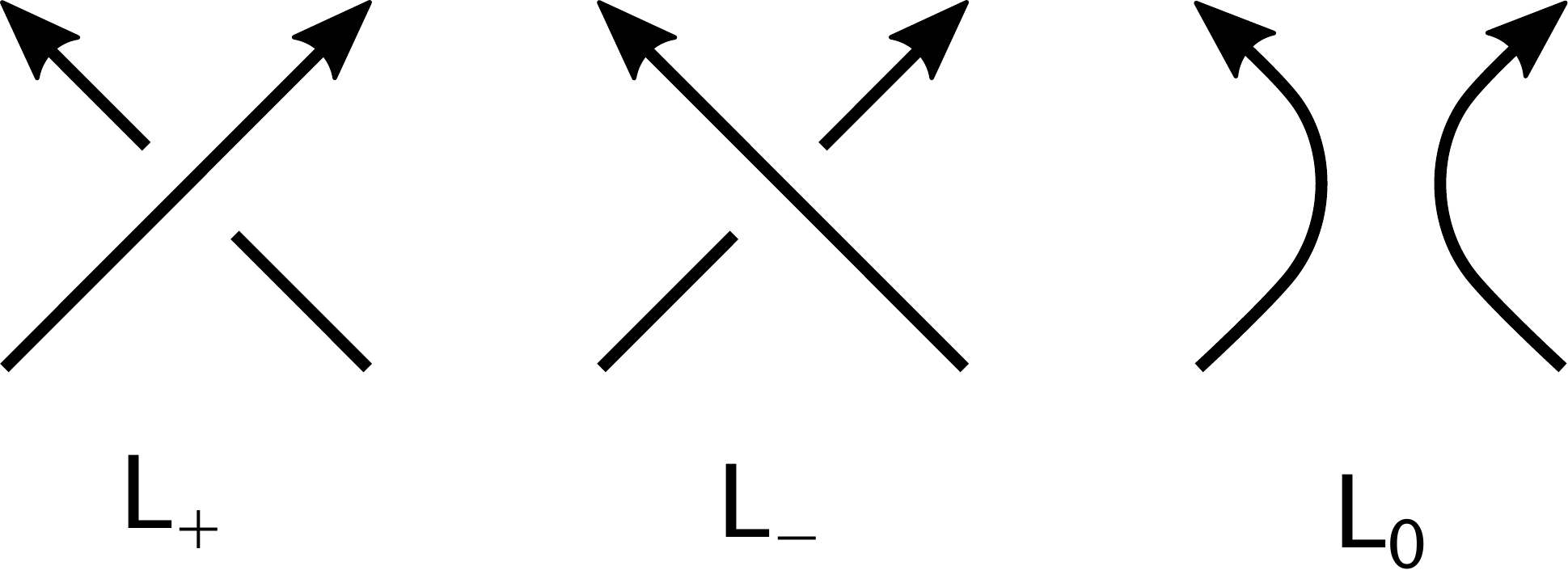}
\caption{The local differences of a links appearing in the skein relation.}
\label{fig:skein}
\end{figure}

If $K$ is a knot, then it is known that the Conway polynomial of $K$ only has terms with even powers:
\[
\nabla_K(z)=1 + \sum_{k=1}^{N} a_{2k}(K)z^{2k}
\]
for some $N$ and coefficients $a_{2k}(K),$ which are integral invariants of $K.$ Moreover, the Conway polynomial behaves analogously to the Alexander polynomial with respect to connect sums. If $J$ and $K$ are knots, then:
\begin{equation}\label{eq:conway_sum}
\nabla_{J\#K}(z) = \nabla_J(z)\cdot\nabla_K(z)
\end{equation}

Similarly, if $L$ is a two-component link, then its Conway polynomial has only odd power terms:
\[
\nabla_L(z)= \sum_{k=1}^{N} a_{2k-1}(L)z^{2k-1}
\]
for some $N.$ Moreover, it is known that $a_1(L)$ is equal to the linking number $lk(L).$
As a consequence of this and the skein relation \eqref{eq:conway_skein}, if $L_+$ and $L_-$ are oriented \emph{knots} such that there are diagrams where the two differ only locally near a crossing as in Figure \ref{fig:skein}, then
\begin{equation}\label{eq:a2_skein}
a_2(K_+)-a_2(K_-) = lk(L_0).
\end{equation}
It follows from \eqref{eq:conway_sum} that $a_2$ is additive under connect sum:
\[
a_2(K_1 \# K_2) = a_2(K_1)+a_2(K_2).
\]

\section{The case of $K_1$}

Here we prove Theorem \ref{thm:main}  for the special case of $n=1;$ namely, that the knot $K_1=T_{2,5}\#T_{-2,3}$ admits infinitely many non-characterizing slopes.

This will demonstrate the general approach which will also be used in the general case, with some modification.

In light of Lemmas \ref{lem:00} and \ref{lem:banded}, we can associate to $K_1$ an unknot $\beta_{K_1}$ in the complement of $K_1$ such that performing 0-framed surgery on both $K_1$ and $\beta_{K_1}$ yields $S^3$ (see Figure \ref{fig:link_sp}).
By  Theorem \ref{thm:BM}, it suffices to show that the unknot $\beta_{K_1}$ is not isotopic to a meridian of $K_1.$ If it were isotopic to a meridian, then, with the orientations as in Figure \ref{fig:link_sp}, such an isotopy would preserve linking number, which is 1 in this case. Notice that performing the indicated crossing changes in Figure \ref{fig:link_sp} yields the link, which we denote $L_1,$ which is the union of $K_1$ with an unknot isotopic to a meridian of $K_1,$ also with linking number 1 (see Figure \ref{fig:L1}). Hence, if $\beta_{K_1}$ were isotopic to a meridian of $K_1,$ then $K_1 \cup \beta_{K_1}$ and $L_1$ would be isotopic as \emph{oriented} links.
\begin{figure}
\centering
\begin{subfigure}{.45\textwidth}
\centering
\includegraphics[width=.75\textwidth]{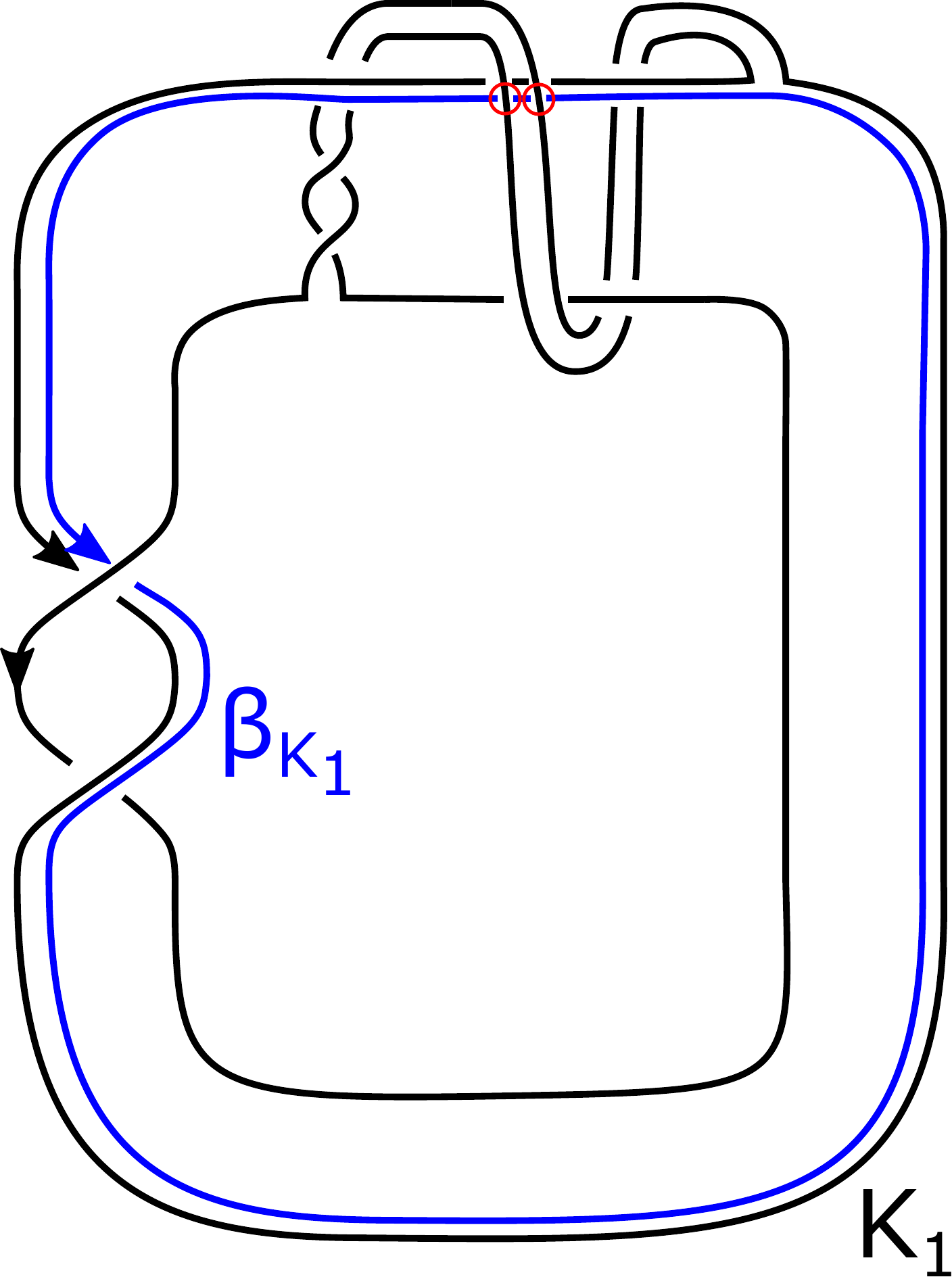}
\caption{The link $K_1 \cup \beta_{K_1}.$}
\label{fig:link_sp}
\end{subfigure}
\begin{subfigure}{.45\textwidth}
\centering
\includegraphics[width=.75\textwidth]{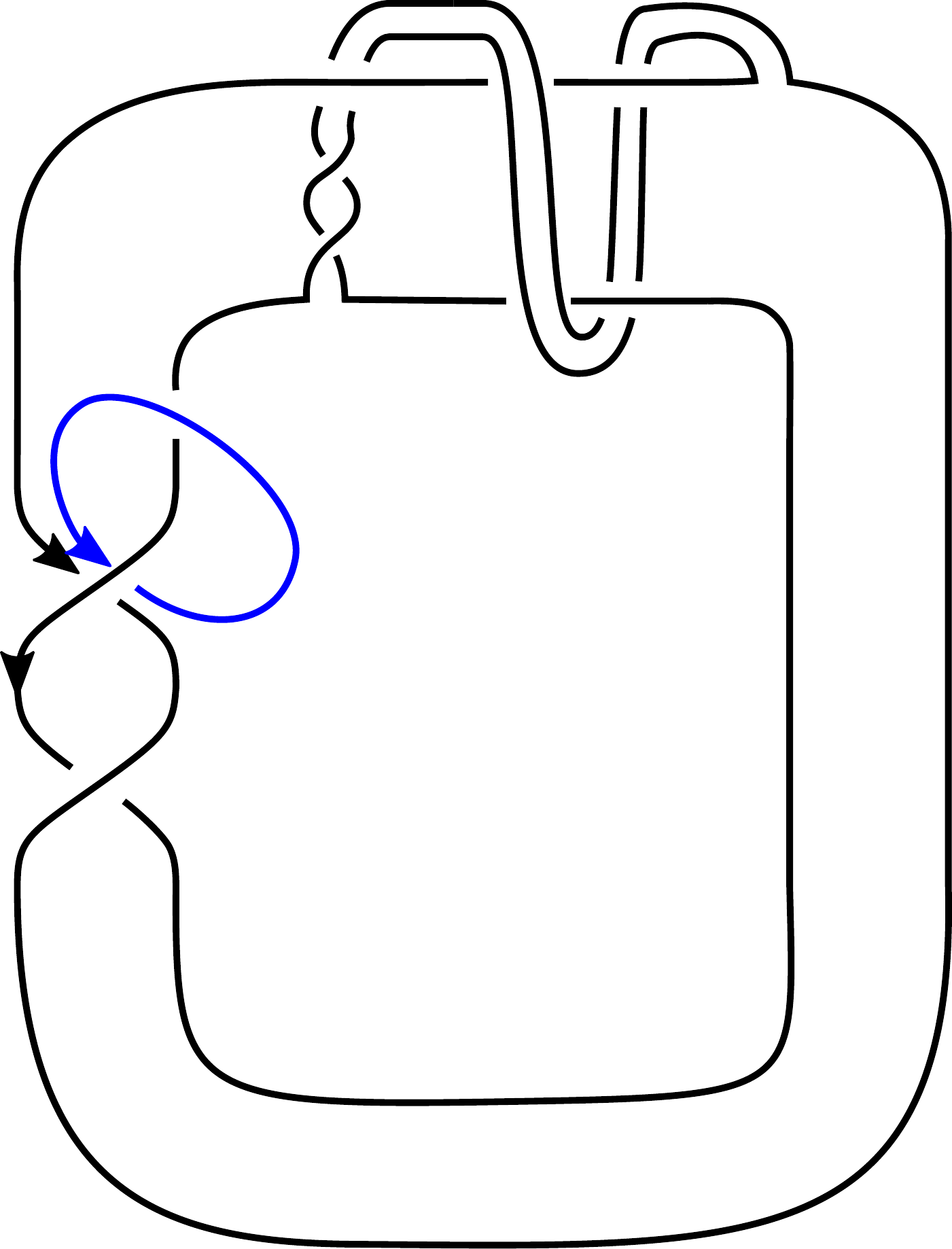}
\caption{The link $L_1,$ which is $K_1$ together with an unknot isotopic to its meridian.}
\label{fig:L1}
\end{subfigure}
\caption{Changing the indicated crossings makes the unknot $\beta_{K_1}$ isotopic to a meridian of $K_1.$}
\end{figure}

We show this is impossible by demonstrating that the two links have different Conway polynomials. Indeed, applying \eqref{eq:conway_skein} to the two crossing changes mentioned above yields:
\begin{equation}\label{eq:conway_K1}
\nabla_{K_1 \cup \beta_{K_1}}(z)-\nabla_{L_1}(z)=z\left(\nabla_{A_1^{0,0}}(z)-\nabla_{B_1^{0,0}}(z)\right).
\end{equation}
Here $A_1^{0,0}$ and $B_1^{0,0}$ are the knots in Figure \ref{fig:A1B1}; the notation follows the more general families of knots given in Section \ref{sec:gen}.
\begin{figure}
\centering
\begin{subfigure}{.45\textwidth}
\centering
\includegraphics[width=.75\textwidth]{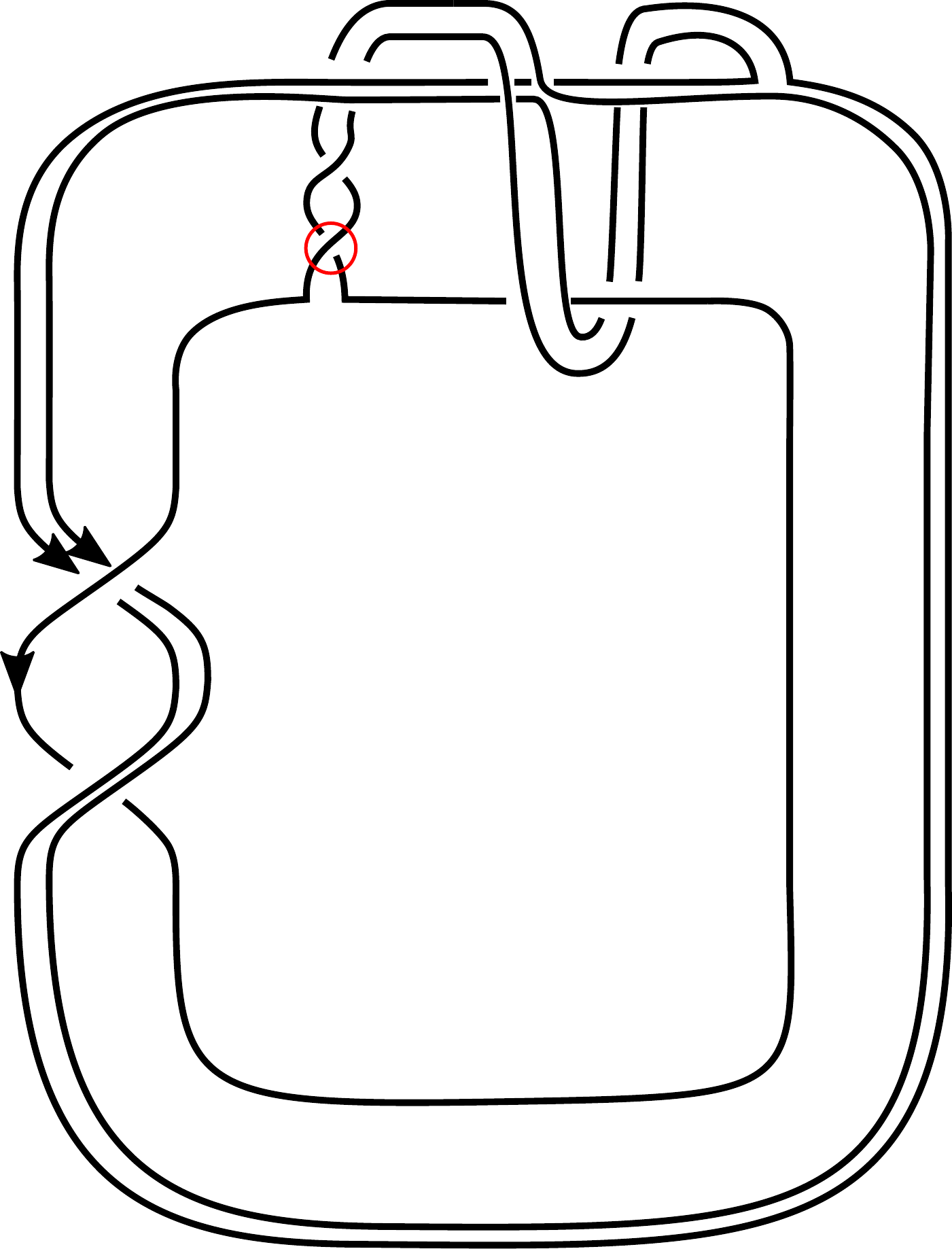}
\caption{$A^{0,0}_1$}
\label{fig:A001_ind}
\end{subfigure}
\begin{subfigure}{.45\textwidth}
\centering
\includegraphics[width=.75\textwidth]{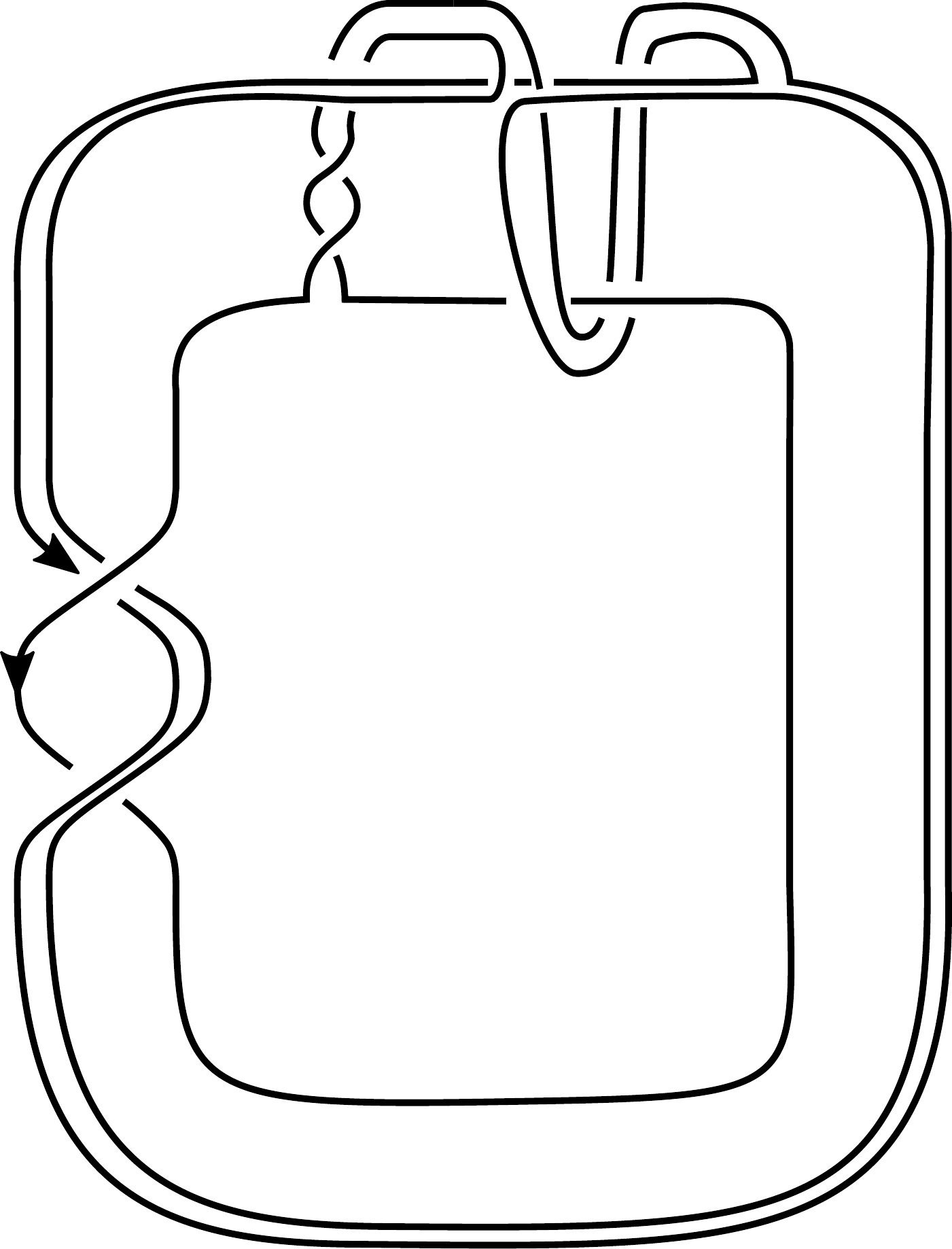}
\caption{$B^{0,0}_1$}
\end{subfigure}
\caption{The knots $A^{0,0}_1$ and $B^{0,0}_1$}
\label{fig:A1B1}
\end{figure}

Now, we may apply the skein relation \eqref{eq:conway_skein} to the crossing of the knot $A_1^{0,0}$ indicated in \ref{fig:A001_ind}. The knot obtained after the crossing change is $A_0^{0,0}$ in our notation from Section \ref{sec:gen}, which is equivalent to the conected sum $8_{19} \# \overline{3}_1$ of two knots in Rolfsen's table \cite{Rol} (see Figure \ref{fig:A_000}). Hence, its Conway polynomial is known. The link obtained after the ``oriented resolution" is $6^2_3$ in Rolfsen's table \cite{Rol} (see Figure \ref{fig:6_3^2}). More precisely, with the orientation we are using, it is the link $L6a1\{1\}$ in \cite{KI}, whence we obtain:
\begin{align*}
\nabla_{A_1^{0,0}}(z)&= \nabla_{A_0^{0,0}}(z) - z\nabla_{6^2_3}(z) \\
&=\nabla_{8_{19} \# \overline{3}_1}(z) - z\nabla_{6^2_3}(z) \\
&=\nabla_{8_{19}}(z) \cdot \nabla_{\overline{3}_1}(z) - z\nabla_{6^2_3}(z) \\
&= 	(1+5z^2+5z^4+z^6)(1+z^2) -z(2z + 2z^3)\\
&=1+4z^2+8z^4+6z^6+z^8
\end{align*}
\begin{figure}
\centering
\includegraphics[width=.6\textwidth]{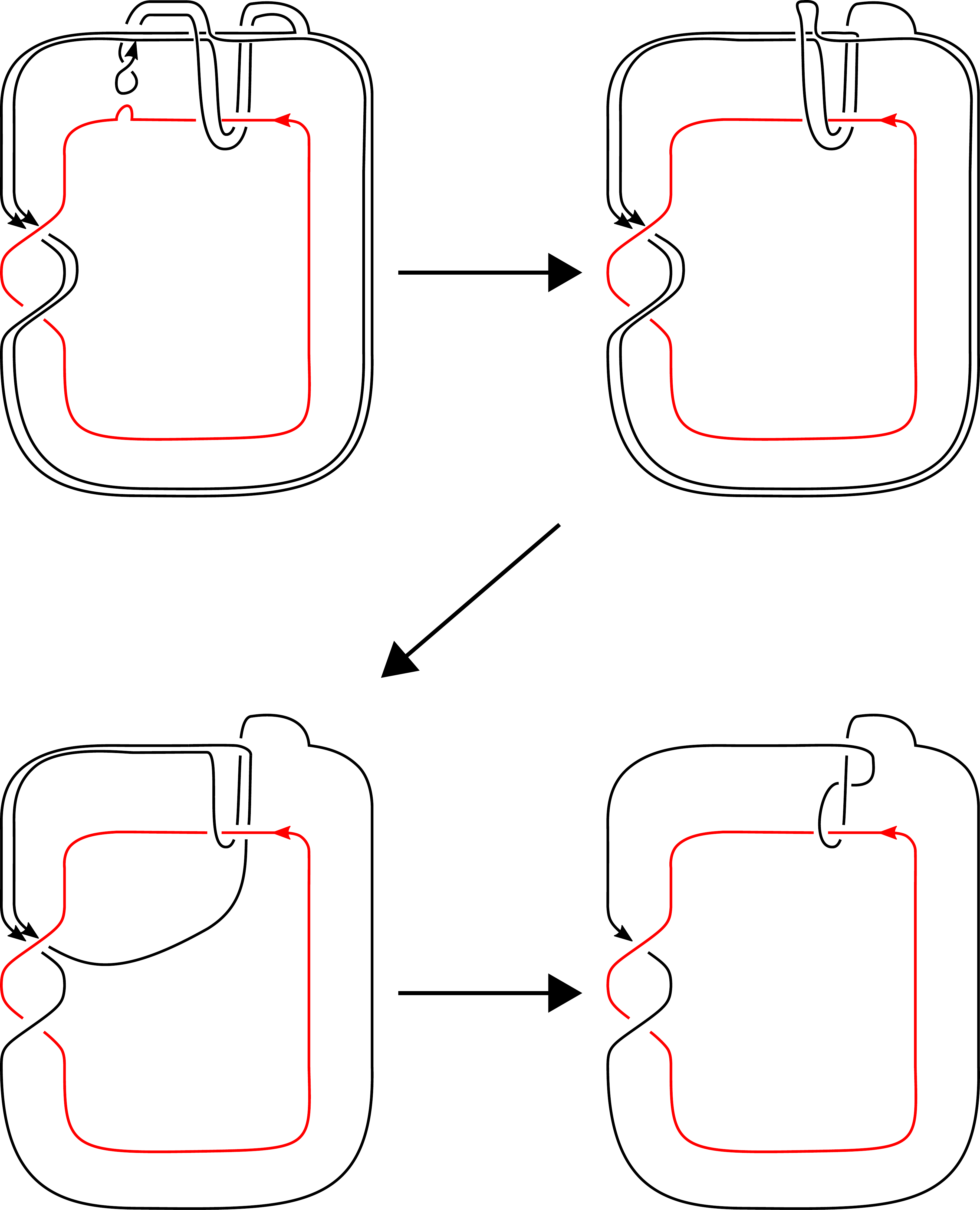}
\caption{Resolving the crossing indicated in Figure \ref{fig:A001_ind} yields a link equivalent to $6_{3}^2.$}
\label{fig:6_3^2}
\end{figure}

\begin{figure}
\centering
\includegraphics[width=.8\textwidth]{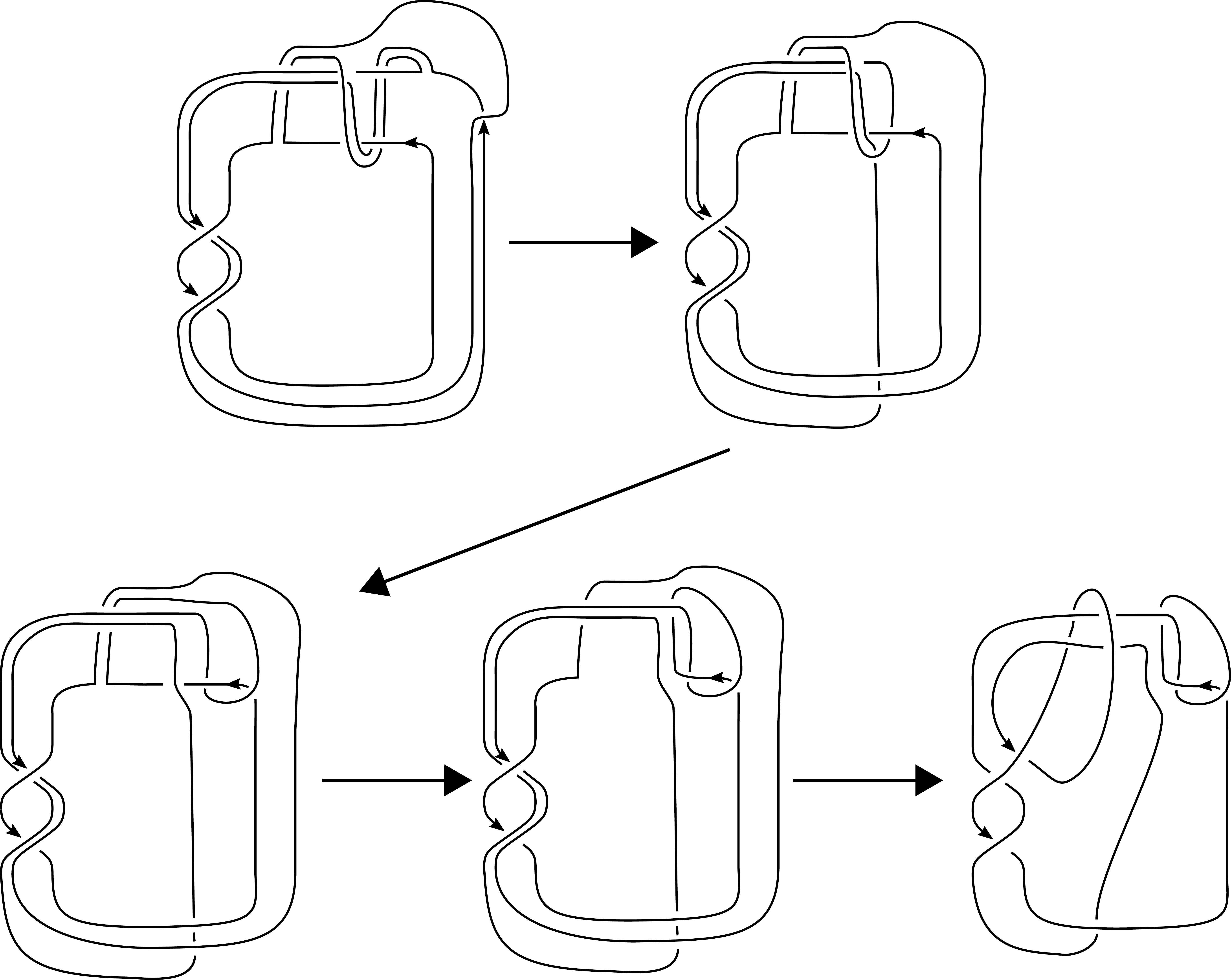}
\caption{The knot $A_0^{0,0}$ is equivalent to $8_{19} \# \overline{3}_1.$}
\label{fig:A_000}
\end{figure}

We observe that $B_1^{0,0}$ is the knot $\overline{10}_{148}$ in the Rolfsen table \cite{Rol} (see Figure \ref{fig:B_001}). Hence, its Conway polynomial is known to be:
\begin{align*}
\nabla_{B_1^{0,0}}(z) &= \nabla_{\overline{10}_{148}}(z) \\
&=1+4z^2+3z^4+z^6
\end{align*}
\begin{figure}
\centering
\includegraphics[width=.8\textwidth]{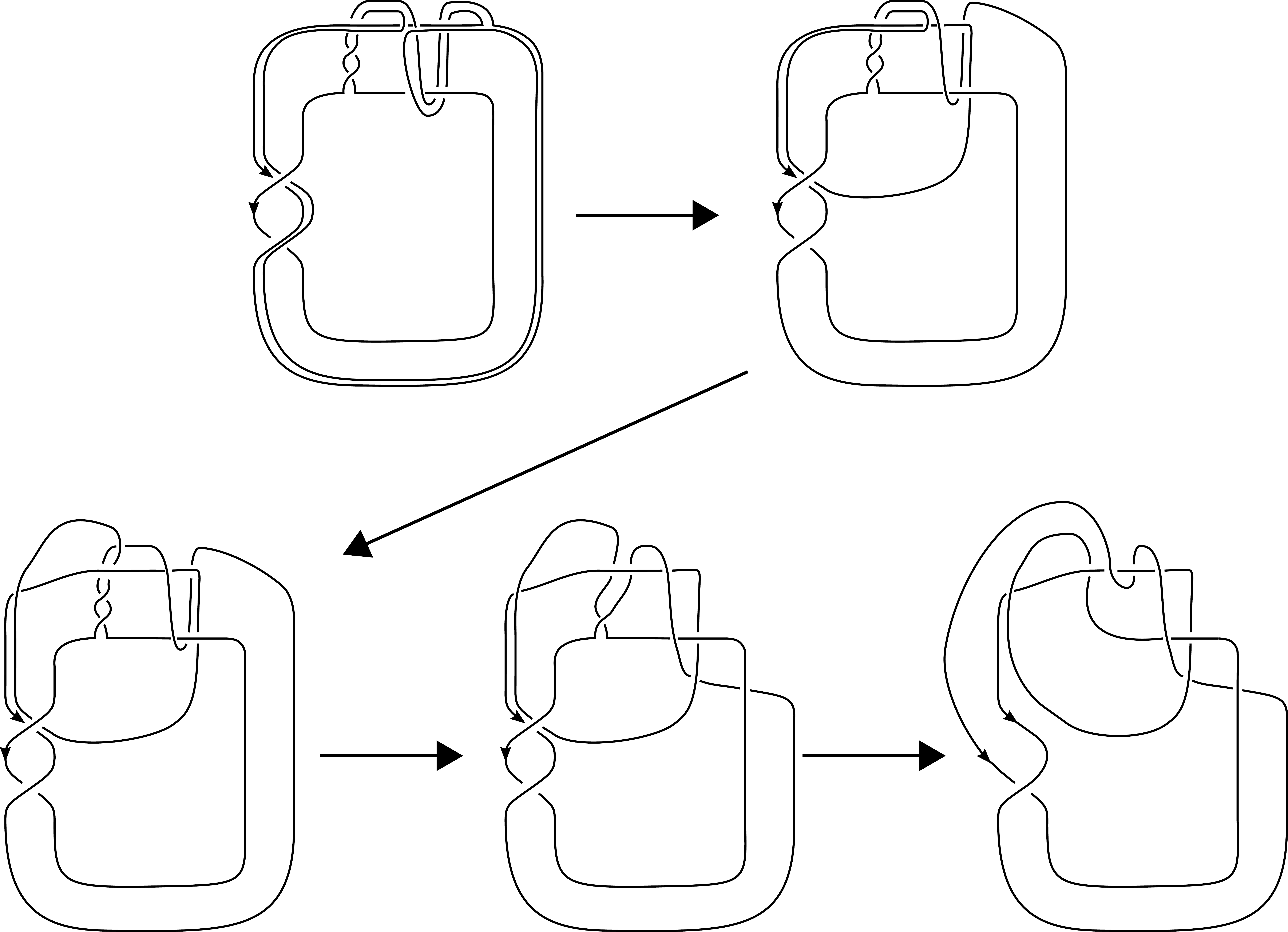}
\caption{The knot $B_1^{0,0}$ is equivalent to $\overline{10}_{148}.$}
\label{fig:B_001}
\end{figure}

By \eqref{eq:conway_K1}, we find:
\begin{align*}
\nabla_{K_1 \cup \beta_{K_1}}(z)-\nabla_{L_1}(z)&=z\left(\nabla_{A_1^{0,0}}(z)-\nabla_{B_1^{0,0}}(z)\right) \\
&=z\left(1+4z^2+8z^4+6z^6+z^8-\left(1+4z^2+3z^4+z^6\right)\right) \\
& \neq 0
\end{align*}
It follows that $\beta_{K_1}$ is not isotopic to a meridian of $K_1,$ and so we have shown:
\begin{lem}\label{lem:K1}
The knot $K_1=T_{2,5}\#T_{-2,3}$ has infinitely many non-characterizing slopes.
\end{lem}

\section{The general case}\label{sec:gen}

We now consider the general case. As in the case of $n=1,$ we wish to show that $\beta_{K_n}$ is not isotopic to a meridian of $K_n$ for $n\geq 2.$ Notice that performing the crossing changes indicated in Figure \ref{fig:link_gen} transforms the link $K_n \cup \beta_{K_n}$ into the link $K_n \cup \mu,$ where $\mu$ is a meridian of $K_n$ whose orientation consistent with that of $\beta_{K_n}$ in the sense that $lk(K_n,\beta_{K_n})=lk(K_n,\mu)=1.$

\begin{figure}
\centering
\includegraphics[width=.5\textwidth]{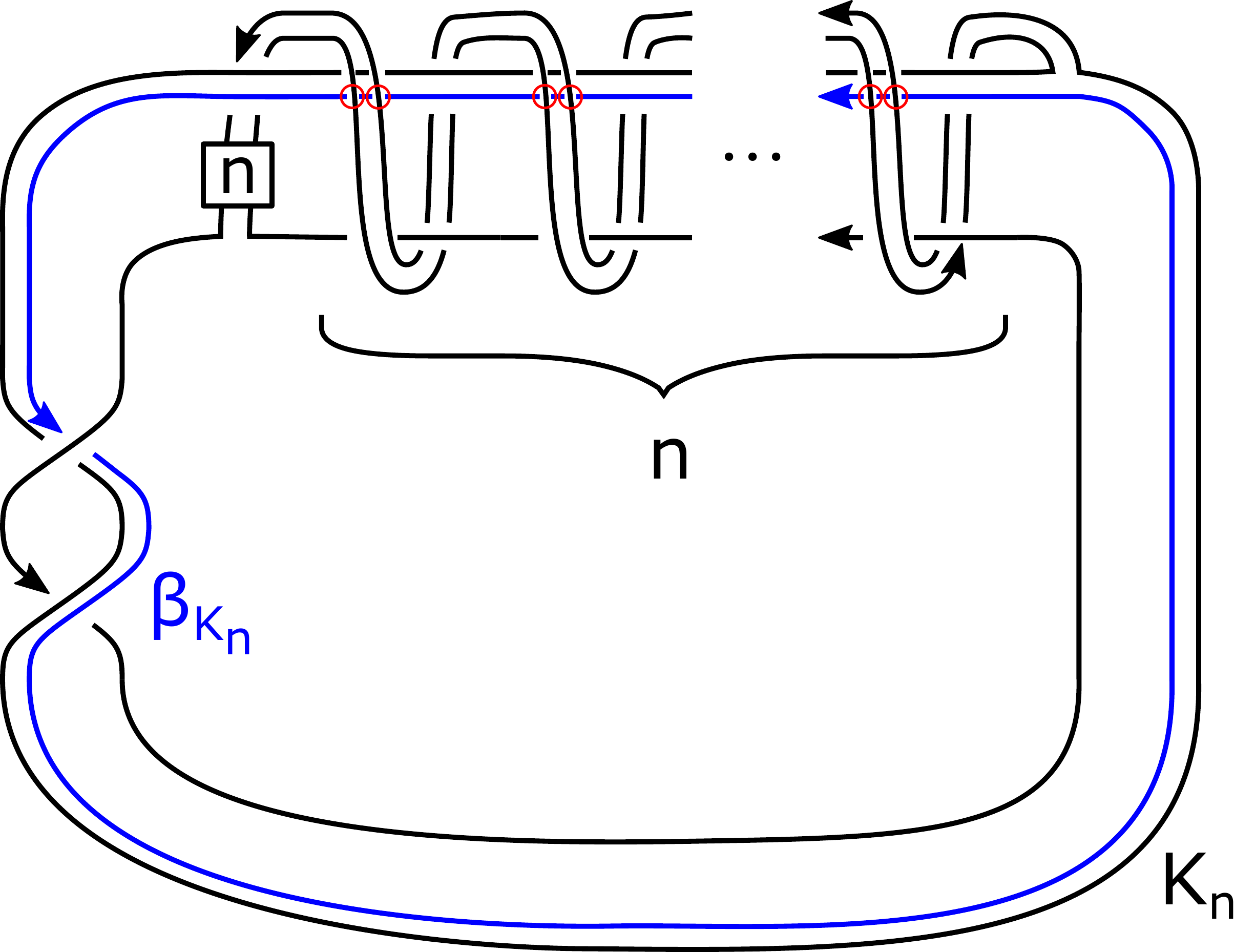}
\caption{Changing the indicated crossings makes the unknot $\beta_{K_n}$ isotopic to a meridian of $K_b.$}
\label{fig:link_gen}
\end{figure}

First, we introduce two families of knots, which we call $A_n^{\ell,r}$ and $B_n^{\ell,r}$ for $n,\ell,r \geq 0.$ These are illustrated in Figures \ref{fig:A_nlr} and \ref{fig:B_nlr}.
\begin{figure}
\centering
\includegraphics[width=.7\textwidth]{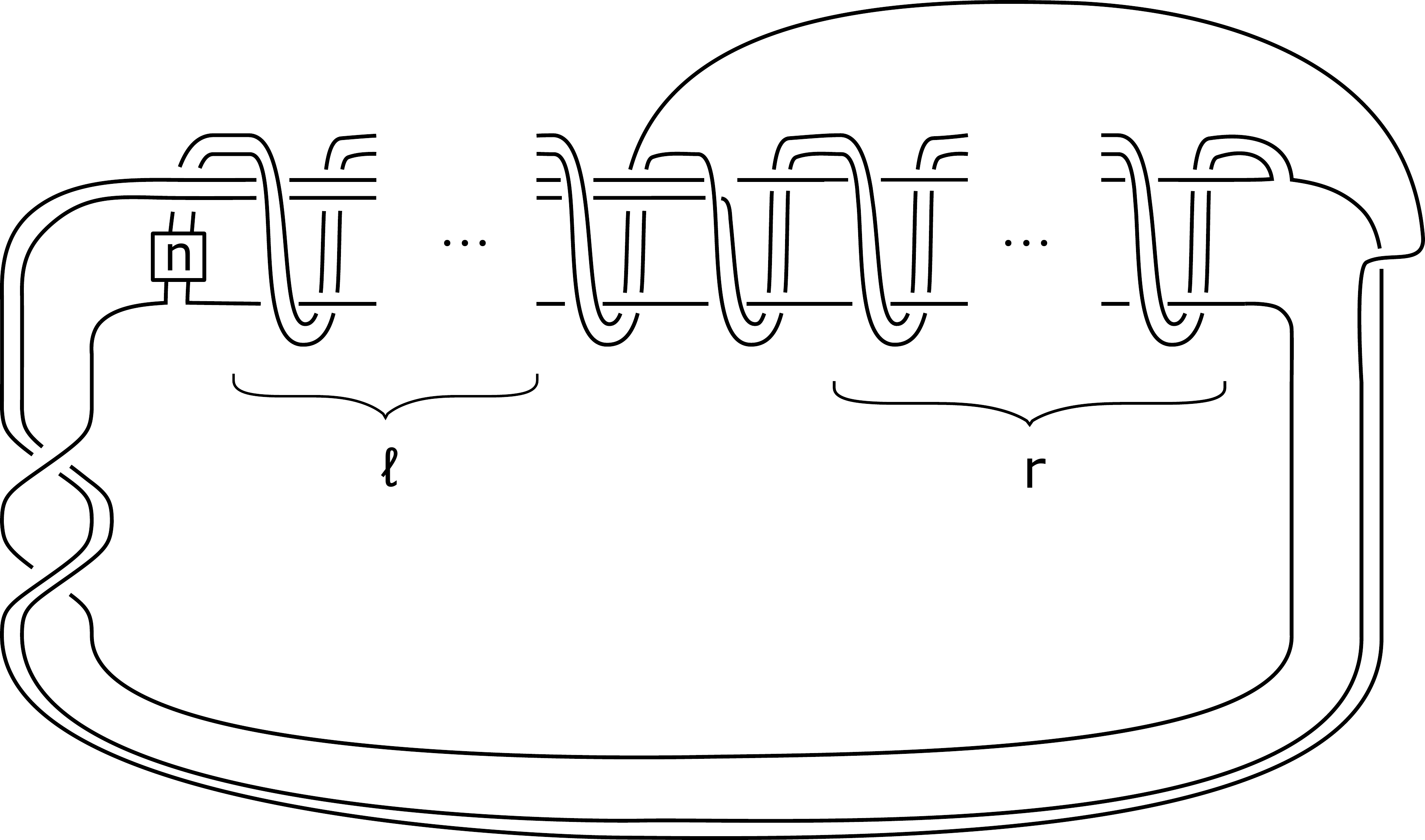}
\caption{The knot $A_{n}^{\ell,r}$}
\label{fig:A_nlr}
\end{figure}

\begin{figure}
\centering
\includegraphics[width=.7\textwidth]{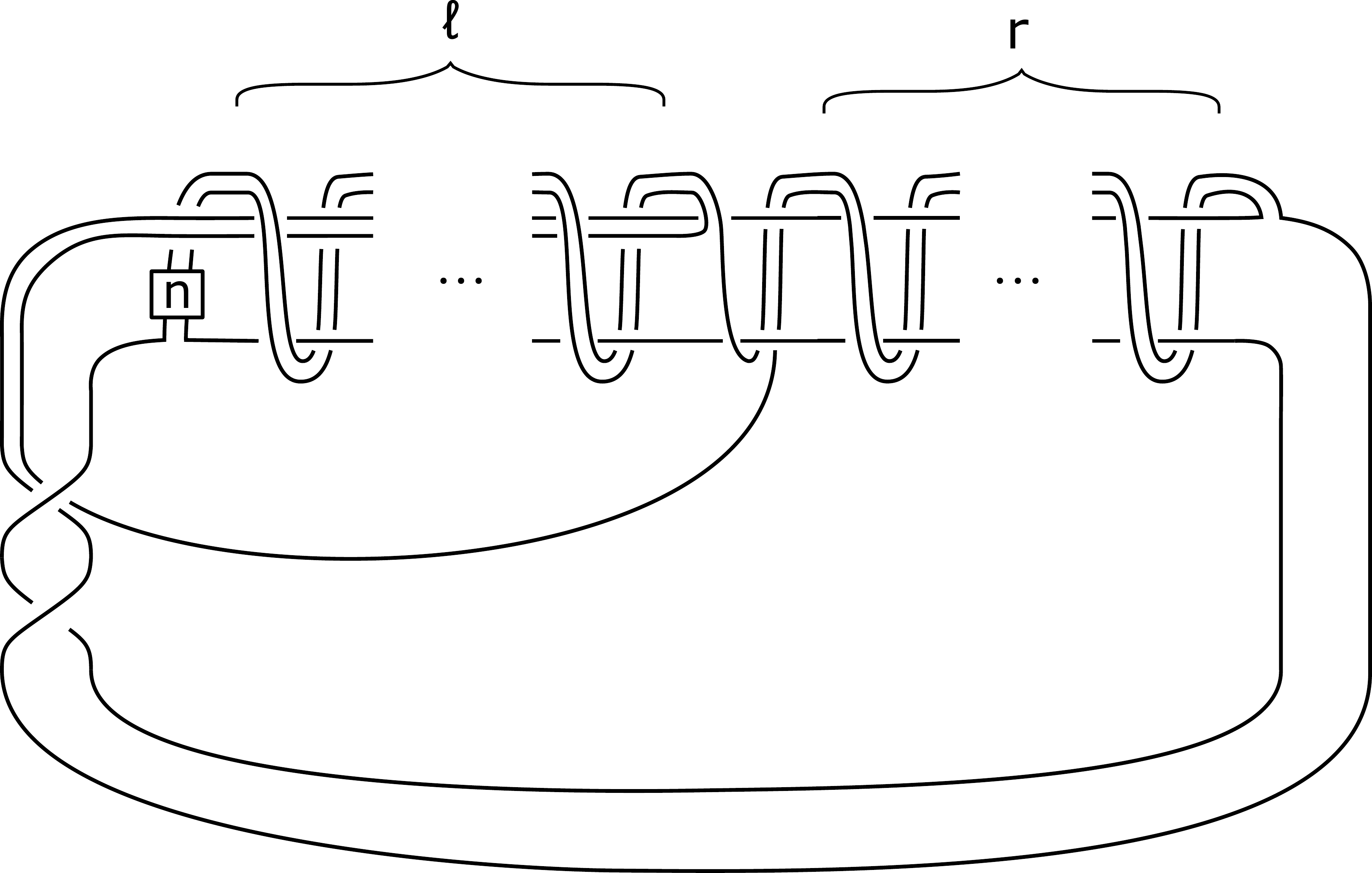}
\caption{The knot $B_{n}^{\ell,r}$}
\label{fig:B_nlr}
\end{figure}

As in the case $n=1,$ we show that $K_n \cup \beta_{K_n}$ is not equivalent to $K_n \cup \mu$ by arguing that the two links have different Conway polynomials. We wish to apply the skein relation \eqref{eq:conway_skein} to change the crossings in Figure \ref{fig:link_gen} from right to left. There are $n$ pairs of adjacent crossings, and so we examine each pair together in a single step. At the $k$th step, the ``oriented resolution" (corresponding to $L_0$ in the notation of Figure \ref{fig:skein}) yields two knots; one for each pair of crossings. Since we are proceeding from right to left, the oriented resolution at the right crossing at step $k$ yields the knot $A_n^{n-k,k-1},$ while the left crossing yields the knot $B_n^{n-k,k-1}.$ Hence, we have:

\begin{lem}
Let $K_n$ and $\beta_{K_n}$ be as in Figure \ref{fig:link_gen}. Let $\mu$ be a meridian of $K_n$ oriented so that $lk(K_n \cup \beta_{K_n})=lk(K_n\cup \mu).$ Then 
\[
\nabla_{K_n\cup\beta_{K_n}}(z) - \nabla_{K_n\cup \mu}(z) = \sum_{k=1}^n \left( \nabla_{A_n^{n-k,k-1}}(z)-\nabla_{B_n^{n-k,k-1}}(z)\right)
\]
\end{lem}

\begin{cor}\label{cor:a3}
Let $K_n$ and $\beta_{K_n}$ be as in Figure \ref{fig:link_gen}. Let $\mu$ be a meridian of $K_n,$ oriented such that $lk(K_n \cup \beta_{K_n})=lk(K_n\cup \mu).$ Then $\nabla_{K_n\cup\beta_{K_n}}(z) - \nabla_{K_n\cup \mu}(z) =  a_3 z^3 + a_5z^5+\dots,$ where
\begin{equation}\label{eq:a3}
a_3= \sum_{k=1}^n \left( a_2\left(A_n^{n-k,k-1}\right)-a_2\left(B_n^{n-k,k-1}\right)\right)
\end{equation}
\end{cor}

\subsection{Computation of Conway polynomial coefficients}

We now compute the invariant $a_2$ for the families of knots $A_n^{\ell,r}$ and $B_n^{\ell,r}$ defined above. This will allow us to explicitly compute the right-hand-side of equation \eqref{cor:a3}.

We first claim:
\begin{lem}\label{lem:A}
$
a_2\left(A_n^{\ell,r}\right)= 4\ell^2+r^2+2\ell r + 6\ell +5r-2n+6
$
\end{lem}
\begin{proof}
We proceed by induction on $\ell,$ $r,$ and $n$ in turn.\\
\textbf{Induction on $\ell$:} We show that $a_2\left(A_0^{\ell,0}\right)= 4\ell^2 + 6\ell +6.$ \\
The base case is $A_0^{0,0},$ which turns out to be the knot $8_{19} \# \overline{3}_1$ (see Figure \ref{fig:A_000}). So its Alexander (and hence Conway) polynomial is known from Rolfsen's table \cite{Rol}, and we compute:
\begin{align*}
a_2(A_0^{0,0}) &= a_2\left(8_{19} \# \overline{3}_1\right)\\
&= a_2(8_{19}) +a_2 \left( \overline{3}_1\right)\\
&=5+1=6.
\end{align*}

Now assume $a_2\left(A_0^{\ell-1,0}\right)=4(\ell-1)^2+6(\ell-1)+6$ for some $\ell \geq 1.$ From Figure \ref{fig:A_0l0}, we see that $A_{0}^{\ell-1,0}$ can be derived from $A_{0}^{\ell,0}$ by performing six crossing changes. Moreover, Figure \ref{fig:cc_A0l0} shows the linking numbers of the links resulting from the oriented resolutions obtained in the after performing the six crossing changes in the sequence indicated. Repeated application of equation \eqref{eq:a2_skein} gives that:
\begin{align*}
a_2\left(A_{0}^{\ell,0}\right) &= a_2\left(A_{0}^{\ell-1,0}\right) - (-2) + 4 -(5-3\ell) + (3\ell-1) - (1-\ell) + (\ell+3)\\
&= a_2\left(A_{0}^{\ell-1,0}\right) + 8\ell + 2\\
&= 4(\ell-1)^2+6(\ell-1)+6 + 8\ell + 2\\
&= 4\ell^2 + 6\ell + 6,
\end{align*}
as desired.
\begin{figure}
\centering
\includegraphics[width=.7\textwidth]{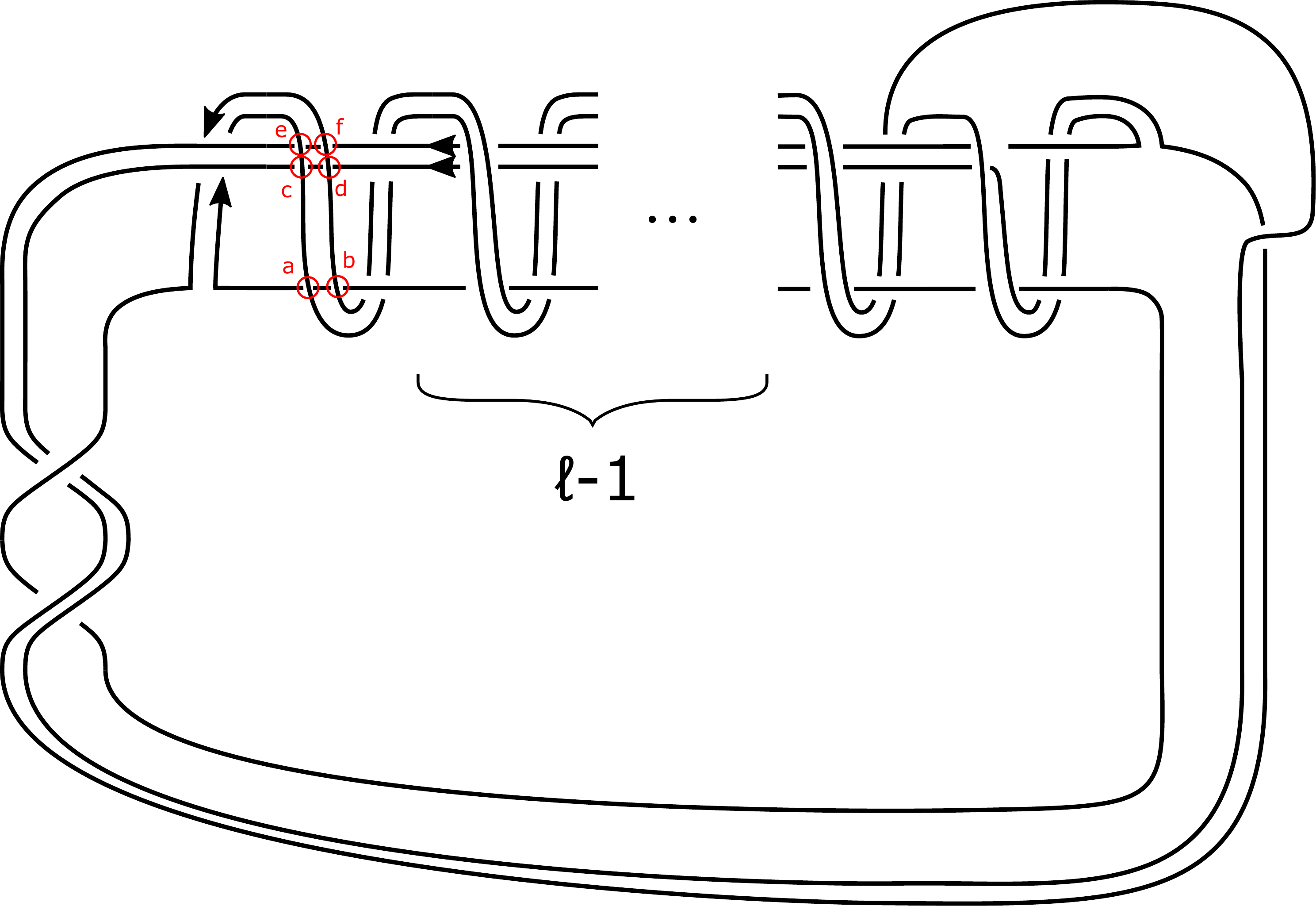}
\caption{The knot $A_{0}^{\ell,0}$. Changing the crossings indicated yields $A_{0}^{\ell-1,0}$.}
\label{fig:A_0l0}
\end{figure}

\begin{figure}
\centering
\begin{subfigure}{.45\textwidth}
\centering
\includegraphics[scale=.15]{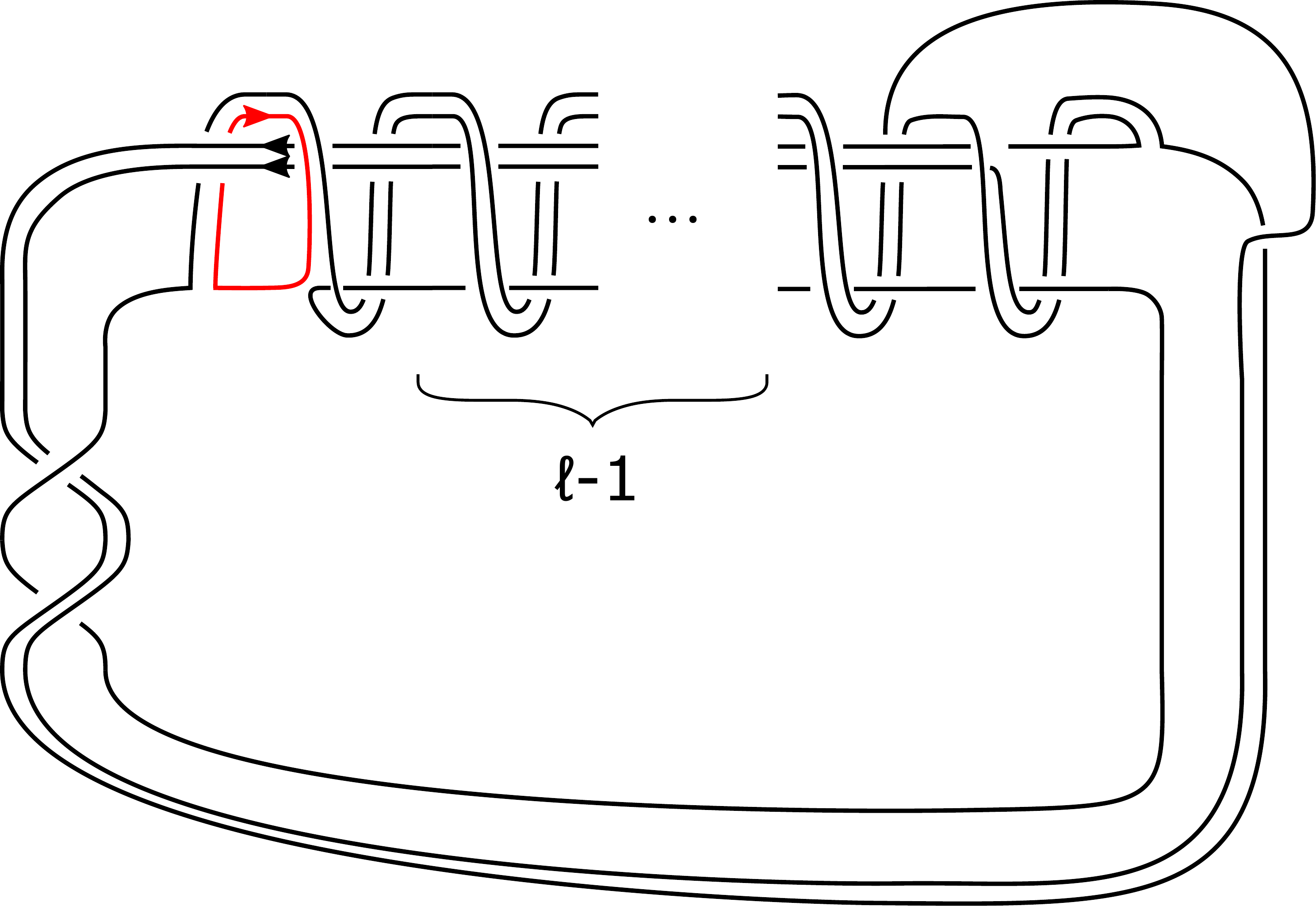}
\caption{linking number $=-2$}
\end{subfigure}
\begin{subfigure}{.45\textwidth}
\centering
\includegraphics[scale=.15]{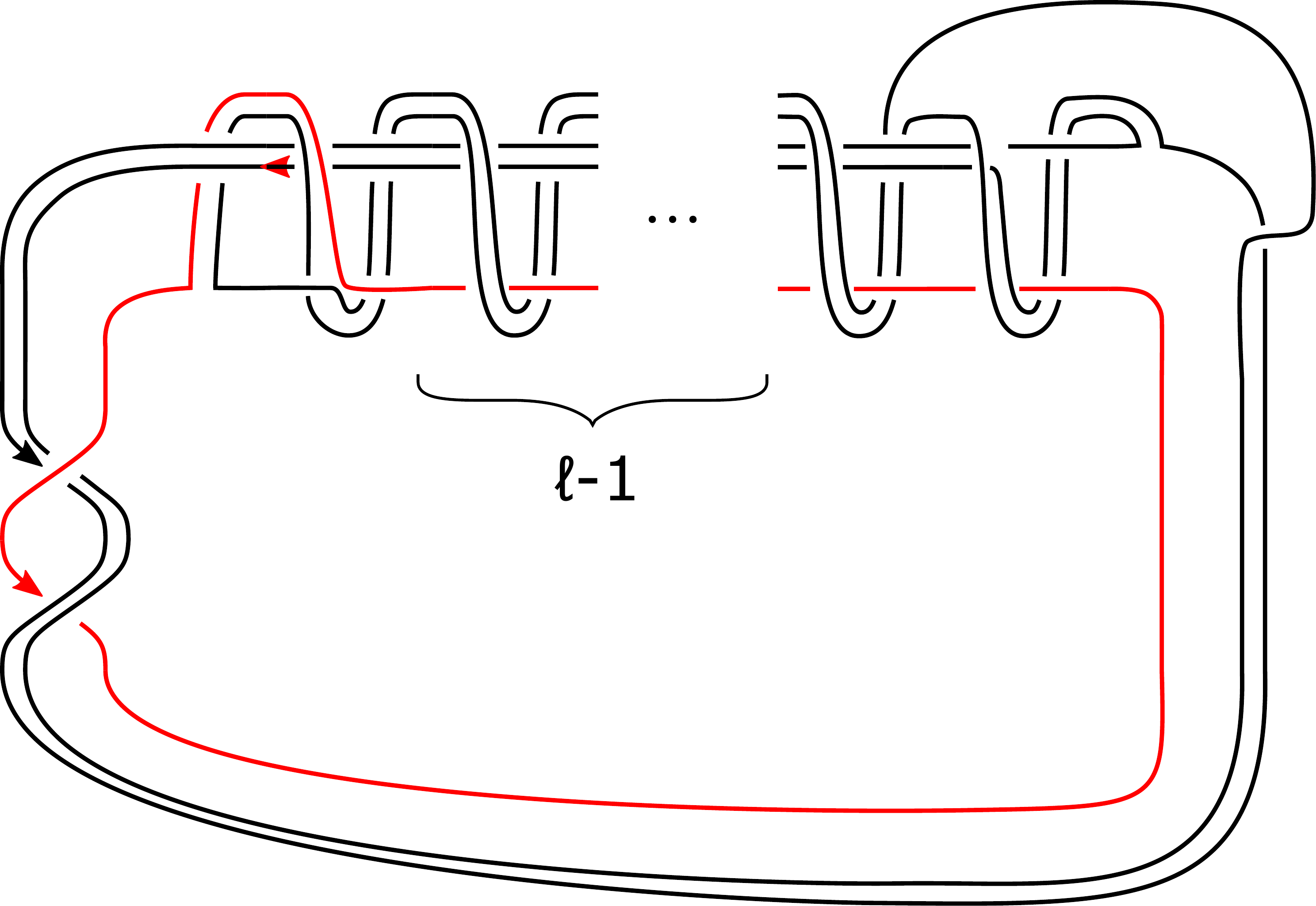}
\caption{linking number $=4$}
\end{subfigure}
\\
\begin{subfigure}{.45\textwidth}
\centering
\includegraphics[scale=.15]{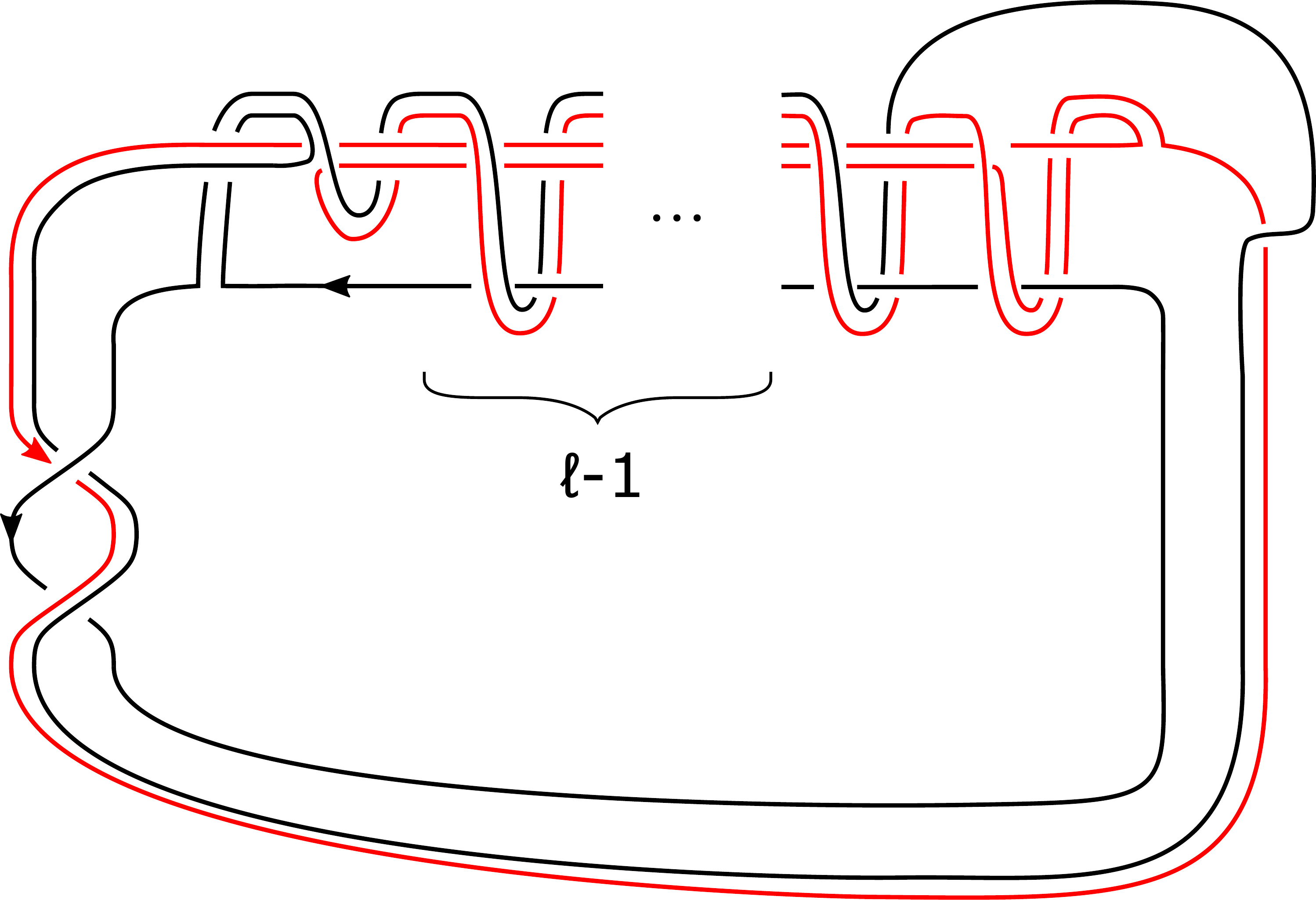}
\caption{linking number $=5-3\ell$}
\end{subfigure}
\begin{subfigure}{.45\textwidth}
\centering
\includegraphics[scale=.15]{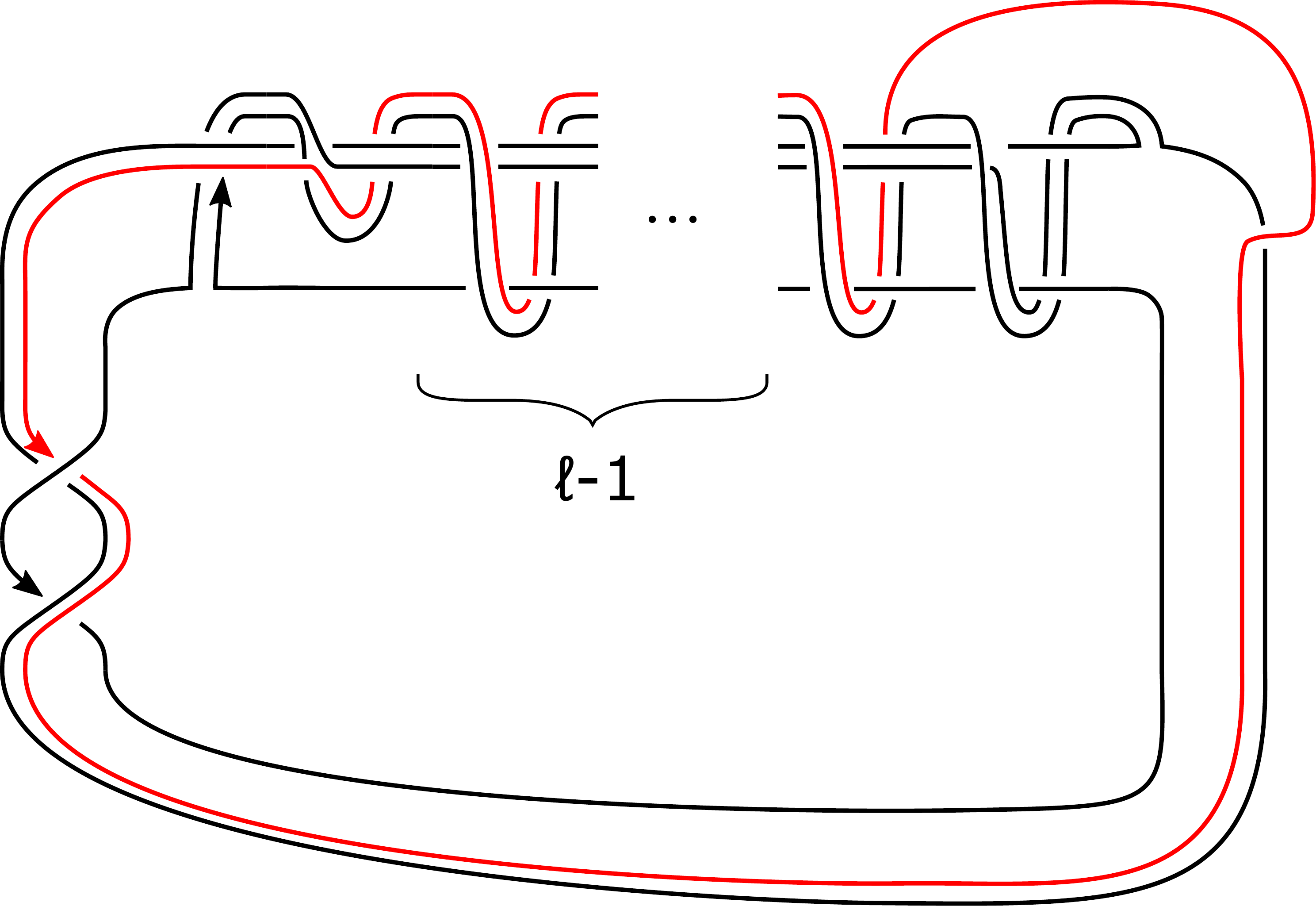}
\caption{linking number $=3\ell-1$}
\end{subfigure}
\\
\begin{subfigure}{.45\textwidth}
\centering
\includegraphics[scale=.15]{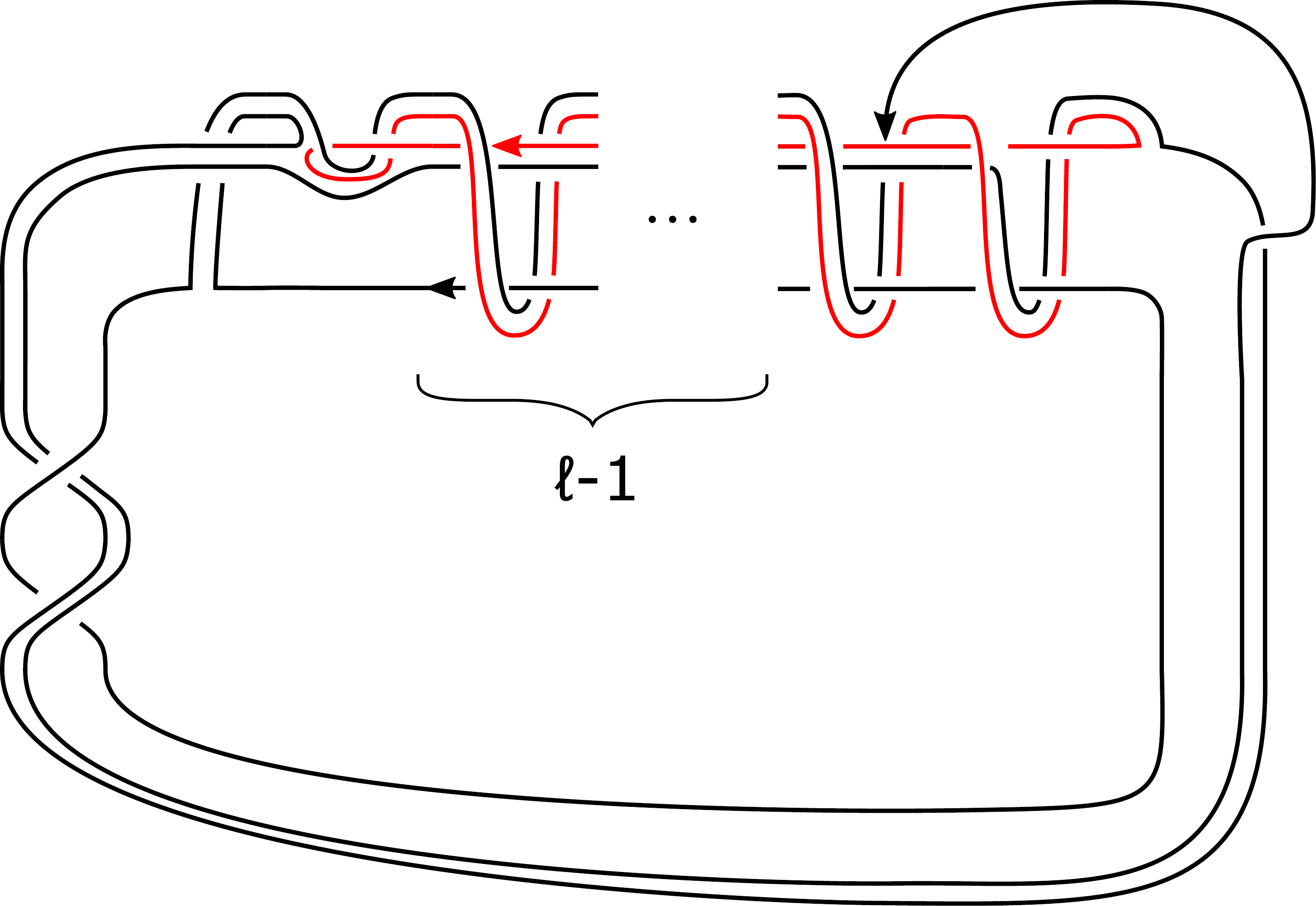}
\caption{linking number $=1-\ell$}
\end{subfigure}
\begin{subfigure}{.45\textwidth}
\centering
\includegraphics[scale=.15]{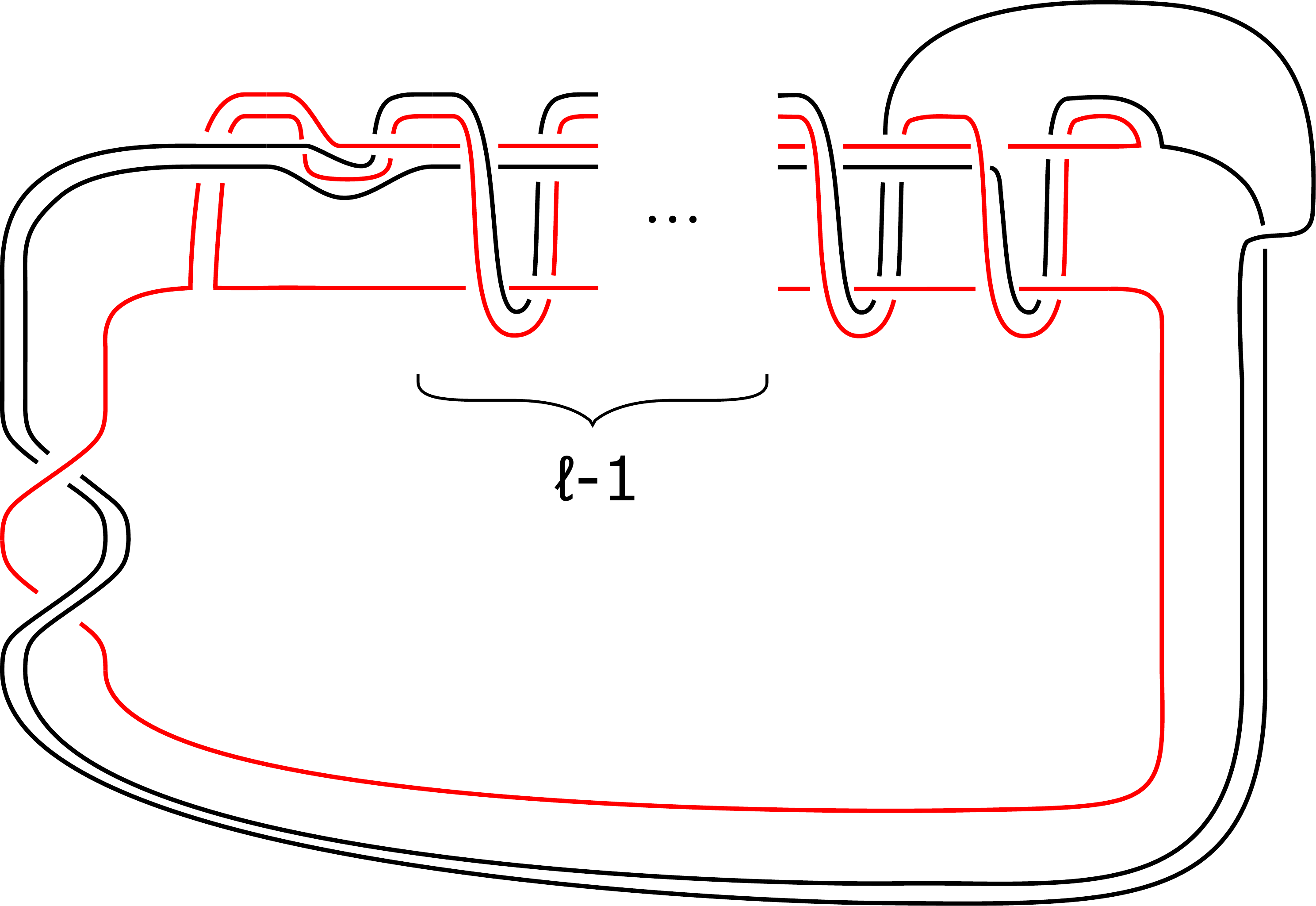}
\caption{linking number $=\ell + 3$}
\end{subfigure}
\caption{Linking numbers of the various links obtained by resolving/changing the crossings in Figure \ref{fig:A_0l0} in sequence.}
\label{fig:cc_A0l0}
\end{figure}
\textbf{Induction on $r$:} We show that $a_2\left(A_0^{\ell,r}\right)= 4\ell^2 +r^2+2\ell r + 6\ell +5r+6.$ \\
The base case ($r=0$) is precisely the conclusion of the previous induction on $\ell.$ So we now assume $a_2\left(A_0^{\ell,r-1}\right)= 4\ell^2 +(r-1)^2+2\ell (r-1) + 6\ell +5(r-1)+6$ for some $r\geq 1.$ We see from Figure \ref{fig:A_0lr} that $A_{0}^{\ell,r-1}$ can be derived from $A_{0}^{\ell,r}$ by performing four crossing changes, and Figure \ref{fig:cc_A0lr} records the linking numbers for the links resulting from the oriented resolutions obtained in the indicated sequence of crossing changes (crossing change d is simply a Reidemeister I move).  Once again, repeatedly applying equation \eqref{eq:a2_skein} gives that:
\begin{align*}
a_2\left(A_{0}^{\ell,r}\right) &= a_2\left(A_{0}^{\ell,r-1}\right) - (-\ell-1) + (\ell+2r+2) +1\\
&= a_2\left(A_{0}^{\ell,r-1}\right) +2\ell+ 2r+4\\
&= 4\ell^2 +(r-1)^2+2\ell (r-1) + 6\ell +5(r-1)+6 +2\ell+ 2r+4\\
&=4\ell^2 + r^2 + 2\ell r + 6\ell +5r +6,
\end{align*}
as desired.
\begin{figure}
\centering
\includegraphics[width=.7\textwidth]{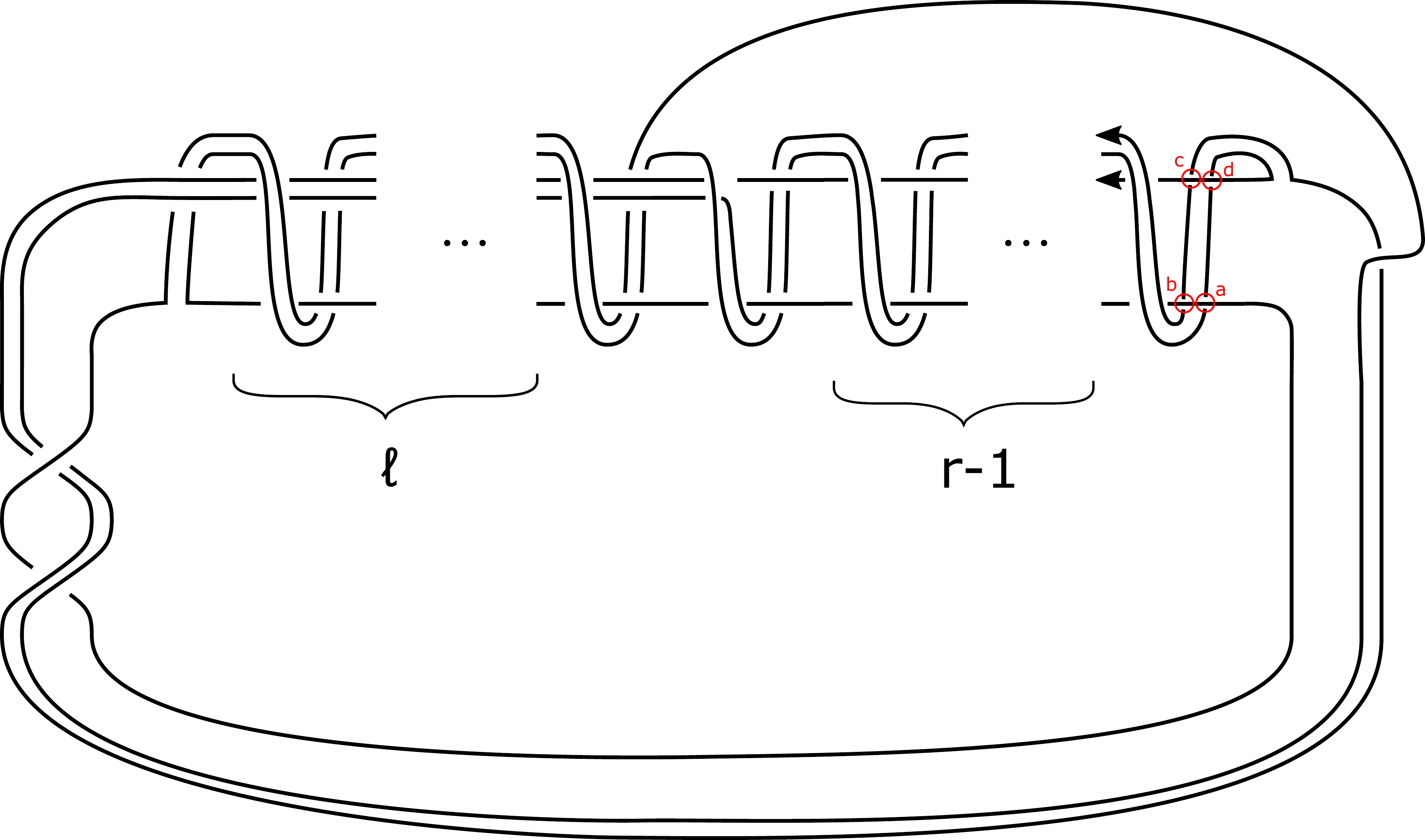}
\caption{The knot $A_{0}^{\ell,r}$. Changing the crossings indicated yields $A_{0}^{\ell,r-1}$.}
\label{fig:A_0lr}
\end{figure}

\begin{figure}
\centering
\begin{subfigure}{.45\textwidth}
\centering
\includegraphics[scale=.15]{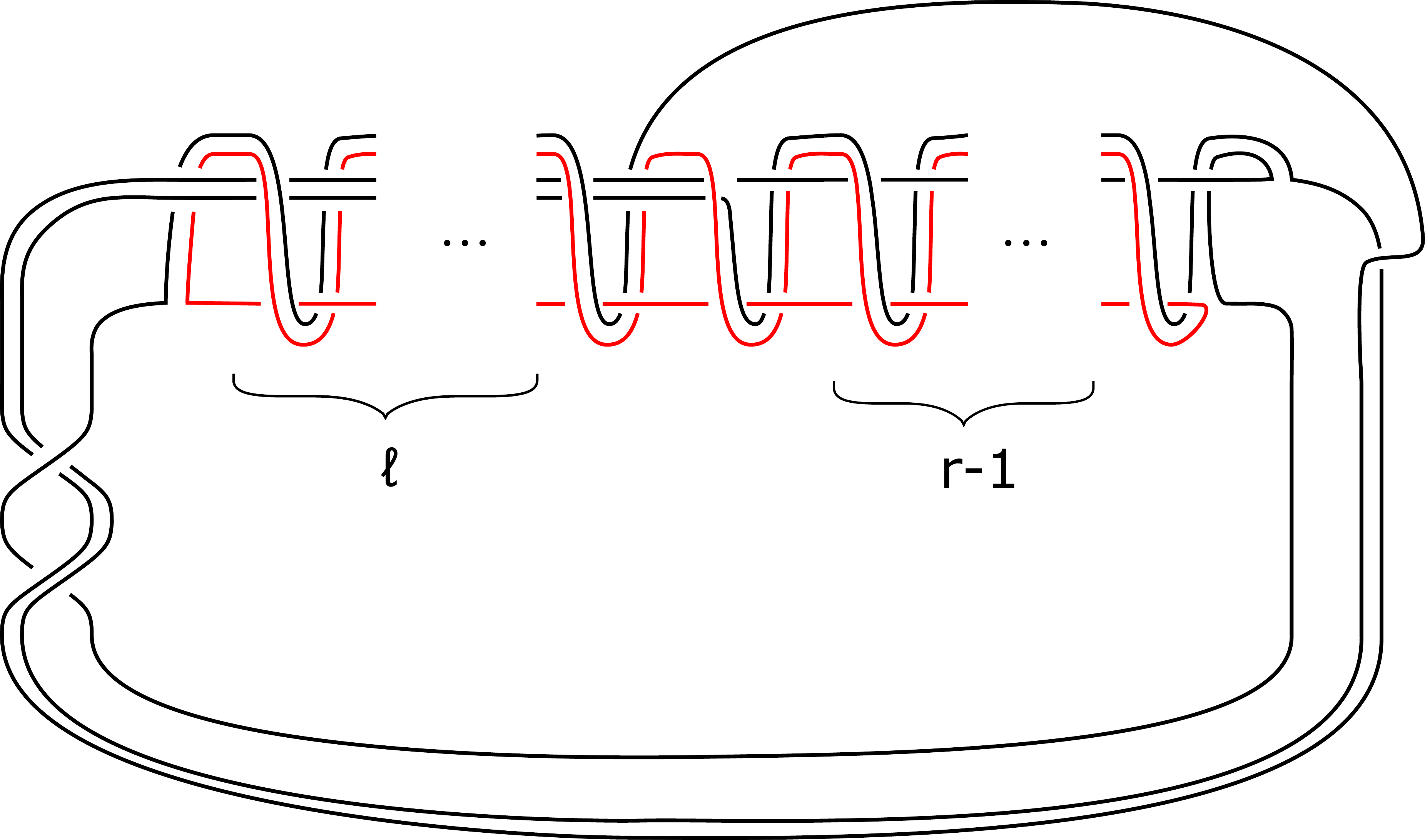}
\caption{linking number $=-\ell-1$}
\end{subfigure}
\begin{subfigure}{.45\textwidth}
\centering
\includegraphics[scale=.15]{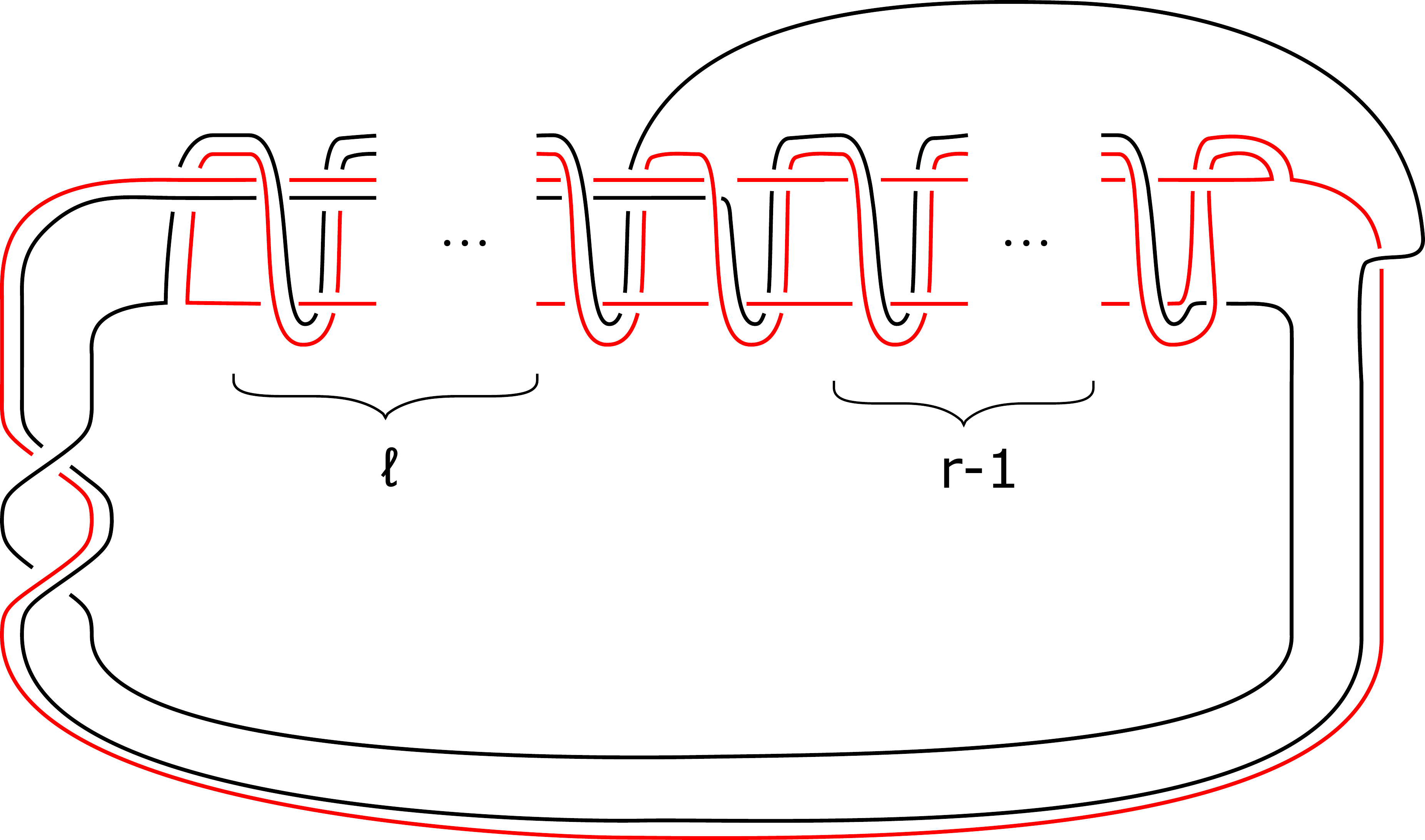}
\caption{linking number $=\ell+2r+2$}
\end{subfigure}
\\
\begin{subfigure}{.45\textwidth}
\centering
\includegraphics[scale=.15]{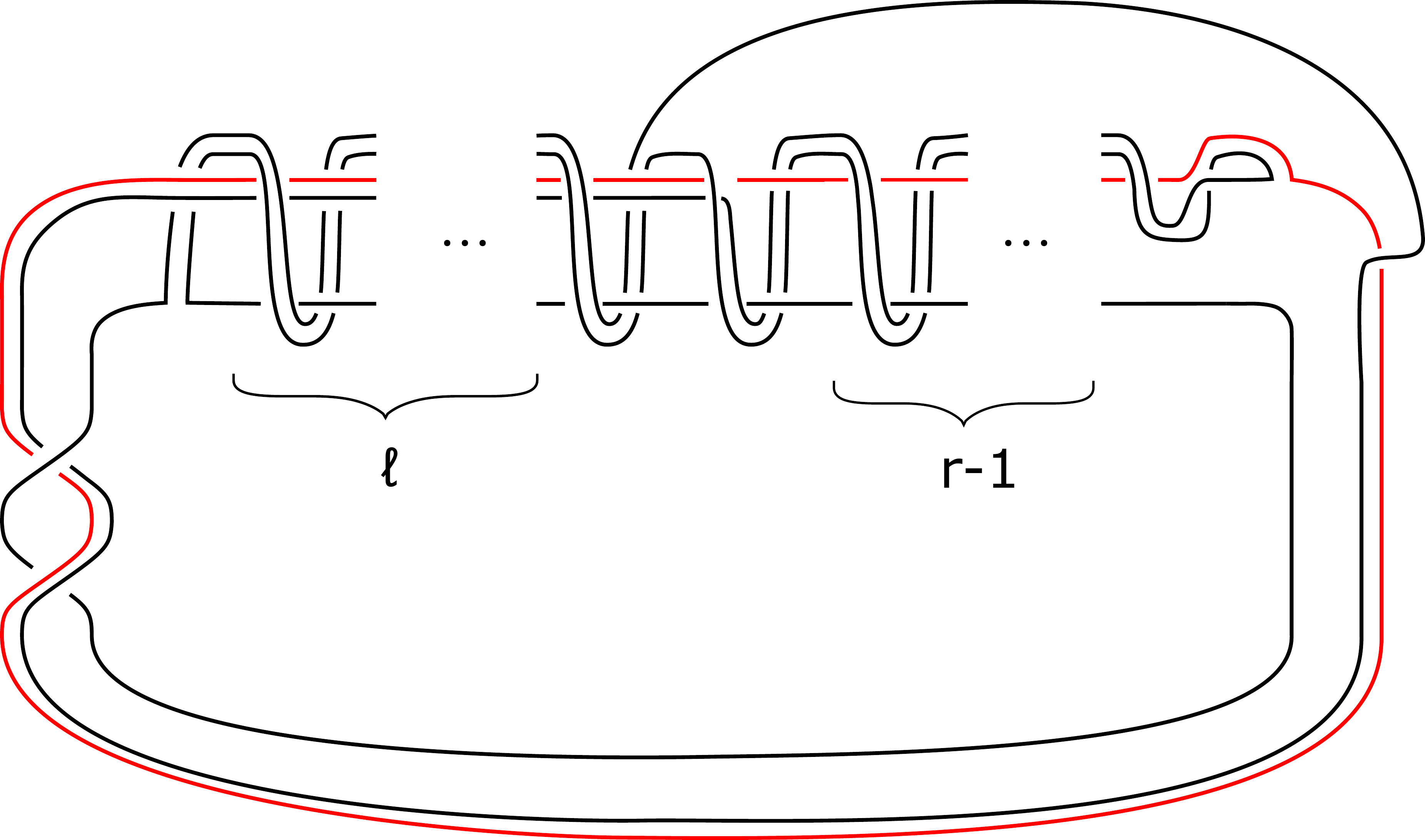}
\caption{linking number $=1$}
\end{subfigure}
\caption{Linking numbers of the various links obtained by resolving/changing the crossings in Figure \ref{fig:A_0lr} in sequence.}
\label{fig:cc_A0lr}
\end{figure}
\textbf{Induction on $n$:} We show that $a_2\left(A_0^{\ell,r}\right)= 4\ell^2 +r^2+2\ell r + 6\ell +5r-2n+6.$ \\
The base case ($n=0$) is precisely the conclusion of the previous induction on $r.$ So we now assume $a_2\left(A_{n-1}^{\ell,r}\right)= 4\ell^2 +(r-1)^2+2\ell r + 6\ell +5r-2(n-1)+6$ for some $n\geq 1.$ We see from Figure \ref{fig:A_nlr_ind} that $A_{n-1}^{\ell,r}$ can be derived from $A_{n}^{\ell,r}$ by performing a crossing change, and Figure \ref{fig:cc_Anlr} records the linking number for the link resulting from the oriented resolution.  Once again, applying equation \eqref{eq:a2_skein} gives that:
\begin{align*}
a_2\left(A_{n}^{\ell,r}\right) &= a_2\left(A_{n-1}^{\ell,r}\right) - 2\\
&=4\ell^2 + r^2 + 2\ell r + 6\ell +5r - 2(n-1) +6 -2\\
&=4\ell^2 + r^2 + 2\ell r + 6\ell +5r - 2n +6,
\end{align*}
as desired.
\begin{figure}
\centering
\begin{subfigure}{.45\textwidth}
\centering
\includegraphics[scale=.15]{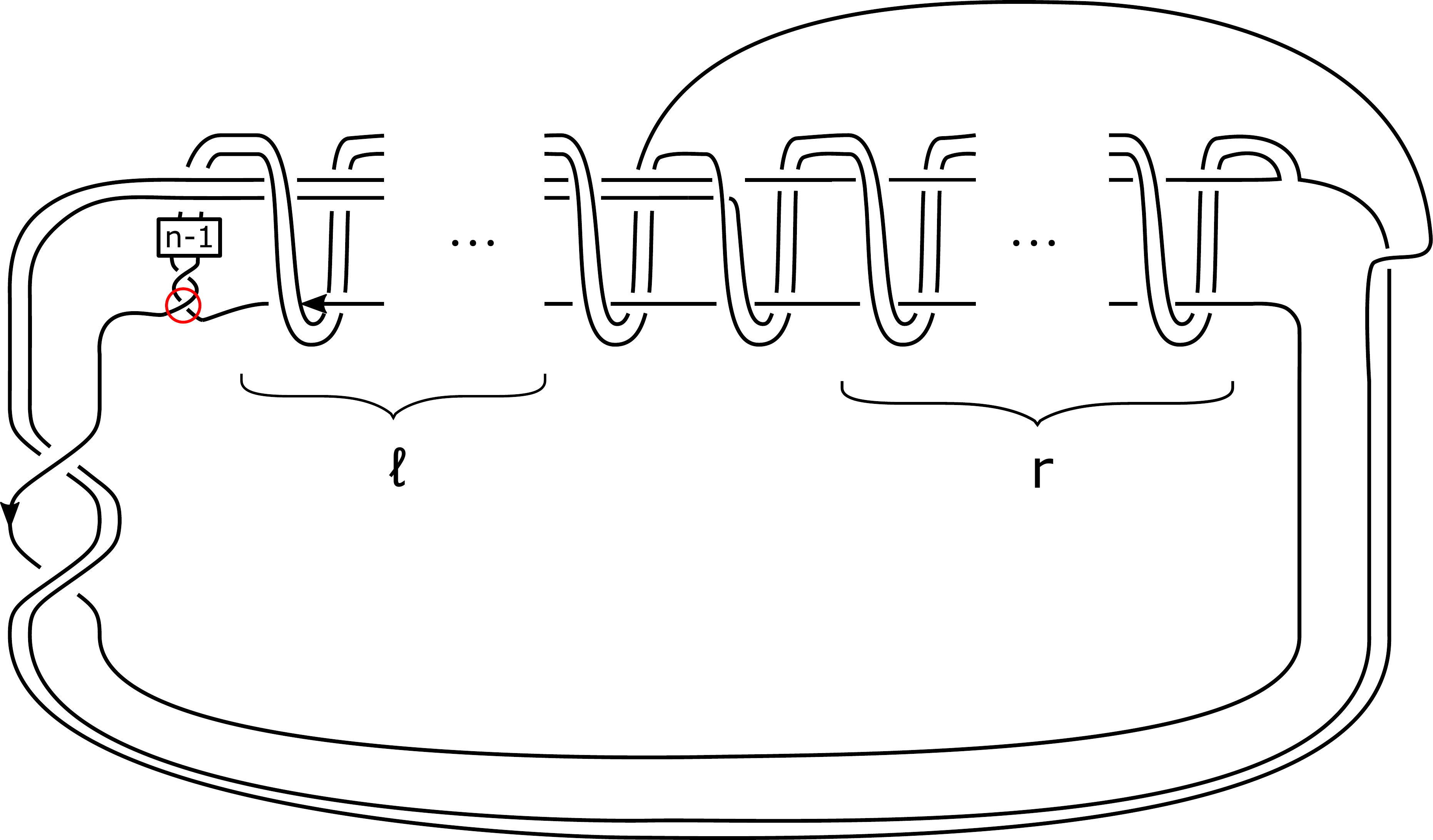}
\caption{The knot $A_{n}^{\ell,r}$. Changing the crossing indicated yields $A_{n-1}^{\ell,r-1}$.}
\label{fig:A_nlr_ind}
\end{subfigure}
\begin{subfigure}{.45\textwidth}
\centering
\includegraphics[scale=.15]{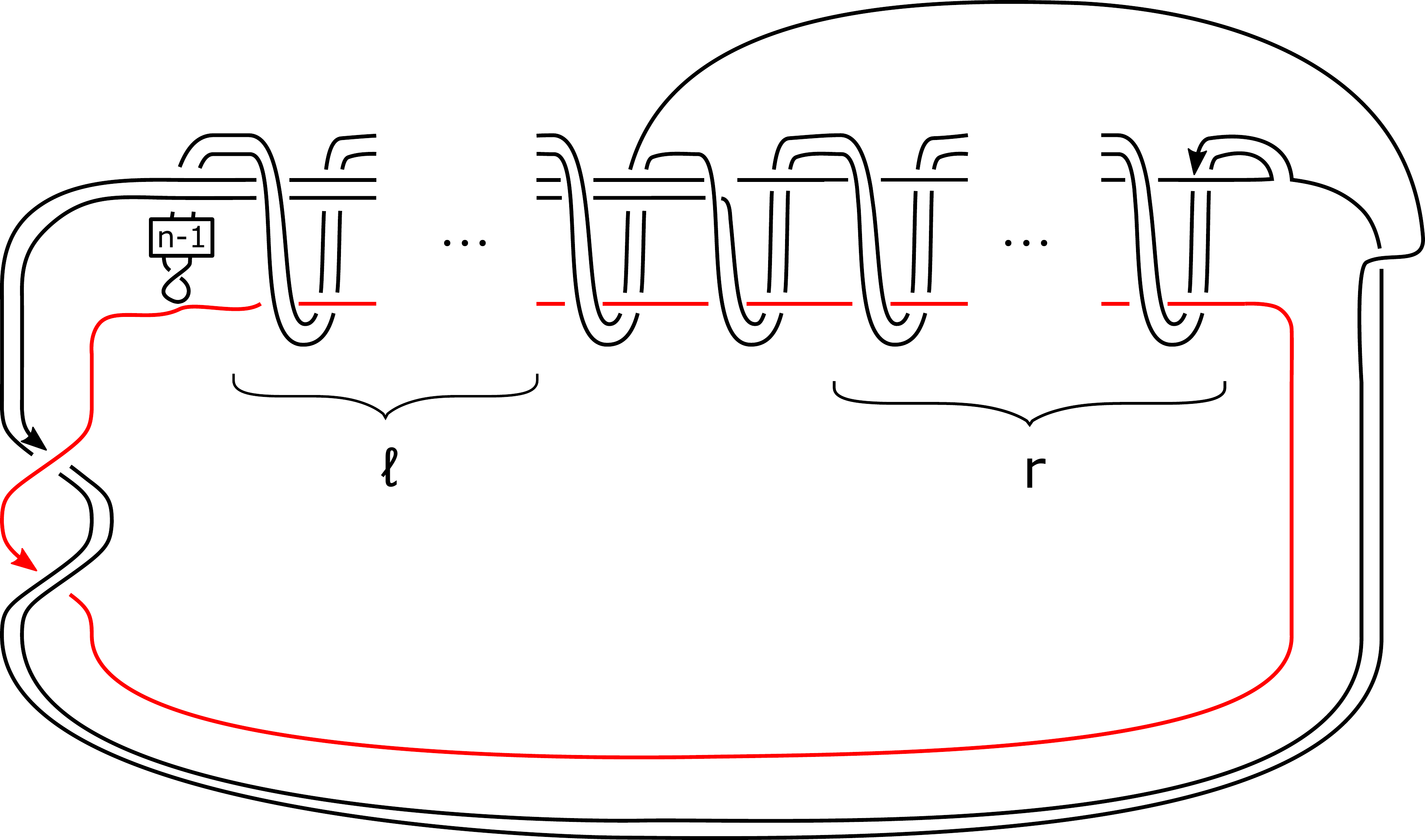}
\caption{linking number $=2$}
\label{fig:cc_Anlr}
\end{subfigure}
\caption{Linking number obtained by resolving the indicated crossing.}
\end{figure}
\end{proof}

Next, we show:

\begin{lem}\label{lem:B}
$
a_2\left(B_n^{\ell,r}\right)= 2\ell^2+r^2+2\ell r + 10\ell +5r-2n+6
$
\end{lem}
\begin{proof}
As in the proof of Lemma \ref{lem:A}, we proceed by induction on $\ell,$ $r,$ and $n$ in turn.\\
\textbf{Induction on $\ell$:} We show that $a_2\left(B_0^{\ell,0}\right)= 2\ell^2 + 10\ell +6.$ \\
The base case is $B_0^{0,0},$ which differs from the knot $\overline{5}_2$ by a single crossing change (see Figure \ref{fig:B000}). We deduce the Conway polynomial of $\overline{5}_2$ from Rolfsen's table \cite{Rol}, and so applying equation \eqref{eq:a2_skein} we obtain:
\[
a_2\left(B_0^{0,0}\right) = a_2\left(\overline{5}_2\right) + 4 = 4+2=6
\]
\begin{figure}
\centering
\begin{subfigure}{.65\textwidth}
\centering
\includegraphics[scale=.15]{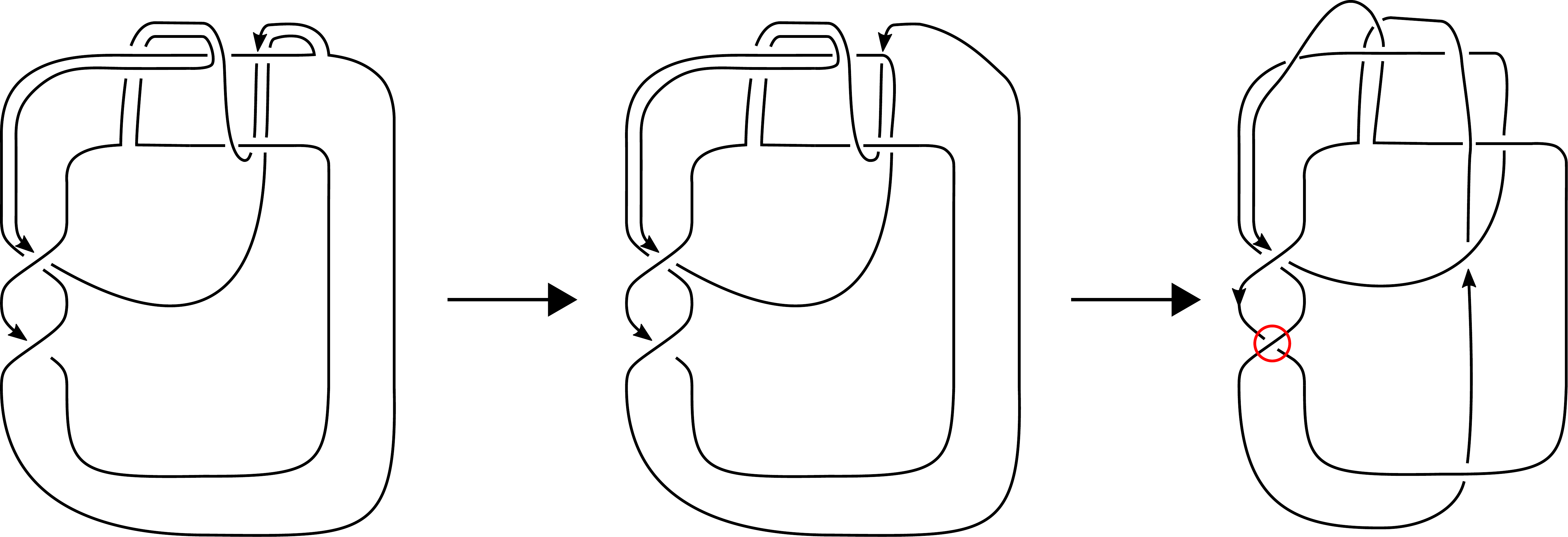}
\caption{Simplifying the diagram of $B_0^{0,0}$}
\end{subfigure}
\begin{subfigure}{.3\textwidth}
\centering
\includegraphics[scale=.15]{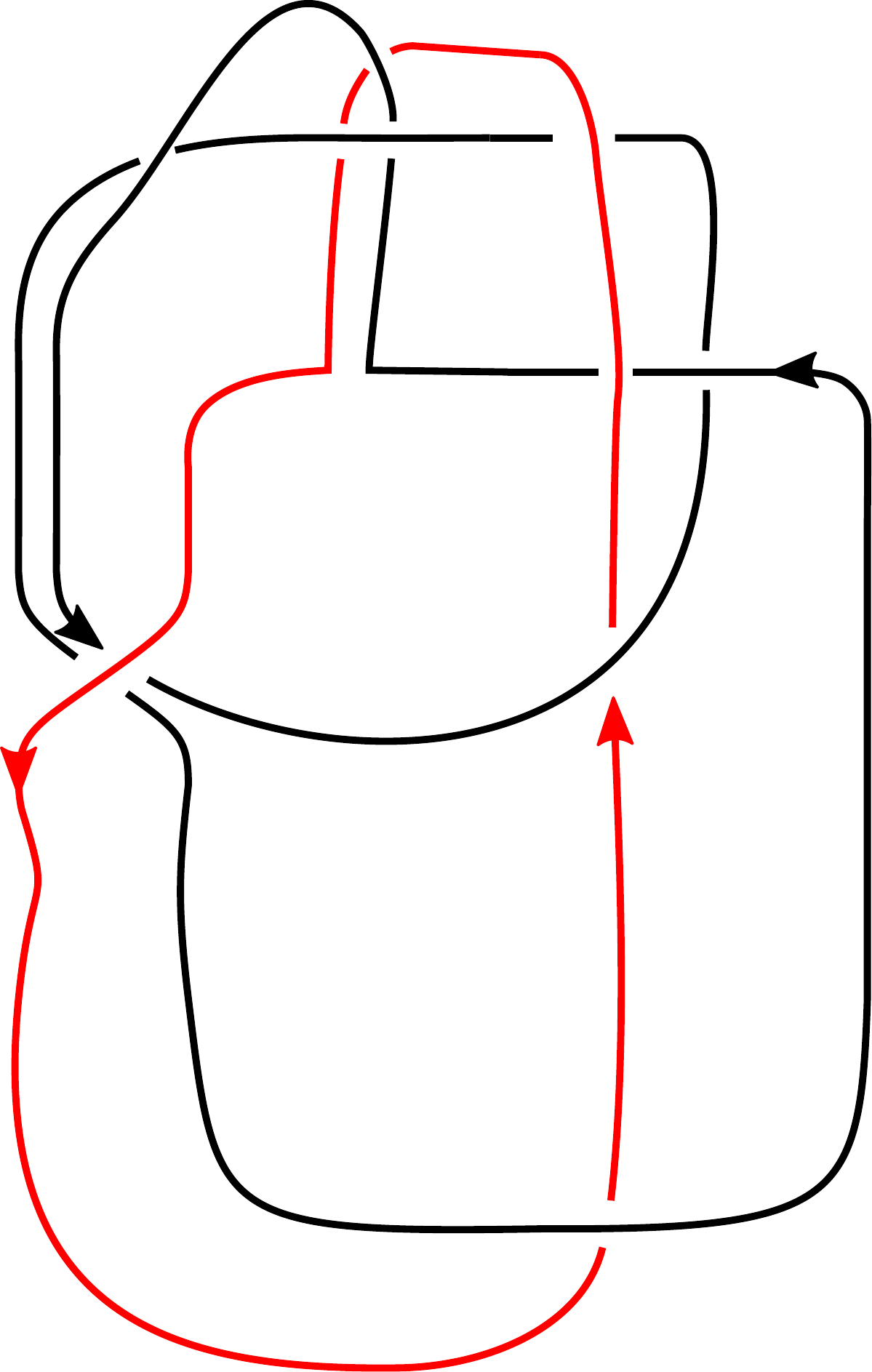}
\caption{linking number $=4$}
\end{subfigure}
\\
\begin{subfigure}{.75\textwidth}
\centering
\includegraphics[scale=.15]{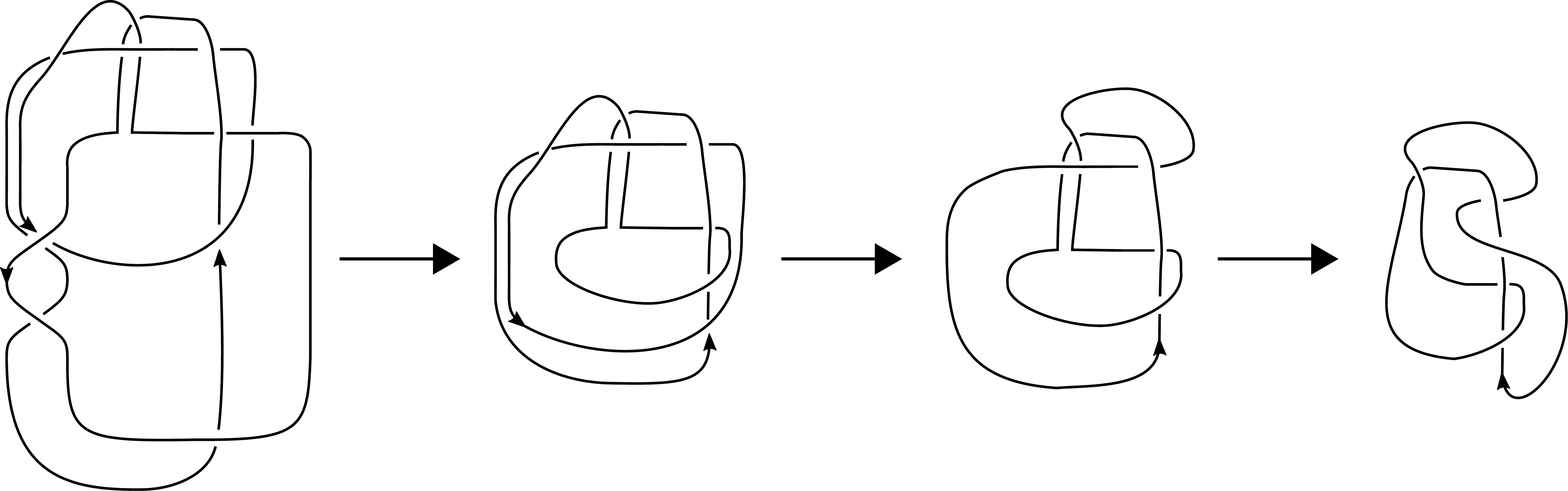}
\caption{After performing the indicated crossing change, the knot becomes $\overline{5}_2.$}
\end{subfigure}
\caption{Applying the skein relation to the knot $B_0^{0,0}.$}
\label{fig:B000}
\end{figure}
Now assume $a_2\left(B_0^{\ell-1,0}\right)=2(\ell-1)^2+10(\ell-1)+6$ for some $\ell \geq 1.$ From Figure \ref{fig:B_0l0}, we see that $B_{0}^{\ell-1,0}$ can be derived from $B_{0}^{\ell,0}$ by performing six crossing changes. Moreover, Figure \ref{fig:cc_B0l0} shows the linking numbers of the links resulting from the oriented resolutions obtained in the after performing the six crossing changes in the sequence indicated. Repeated application of equation \eqref{eq:a2_skein} gives that:
\begin{align*}
a_2\left(B_{0}^{\ell,0}\right) &= a_2\left(B_{0}^{\ell-1,0}\right) - (-2) + 4 -(-\ell-1) + (\ell+1) - (\ell+1) + (3\ell+1)\\
&= a_2\left(B_{0}^{\ell-1,0}\right) +  4\ell + 8\\
&= 2(\ell-1)^2+10(\ell-1)+6 + 4\ell + 8\\
&= 2\ell^2 + 10\ell + 6,
\end{align*}
as desired.
\begin{figure}
\centering
\includegraphics[width=.7\textwidth]{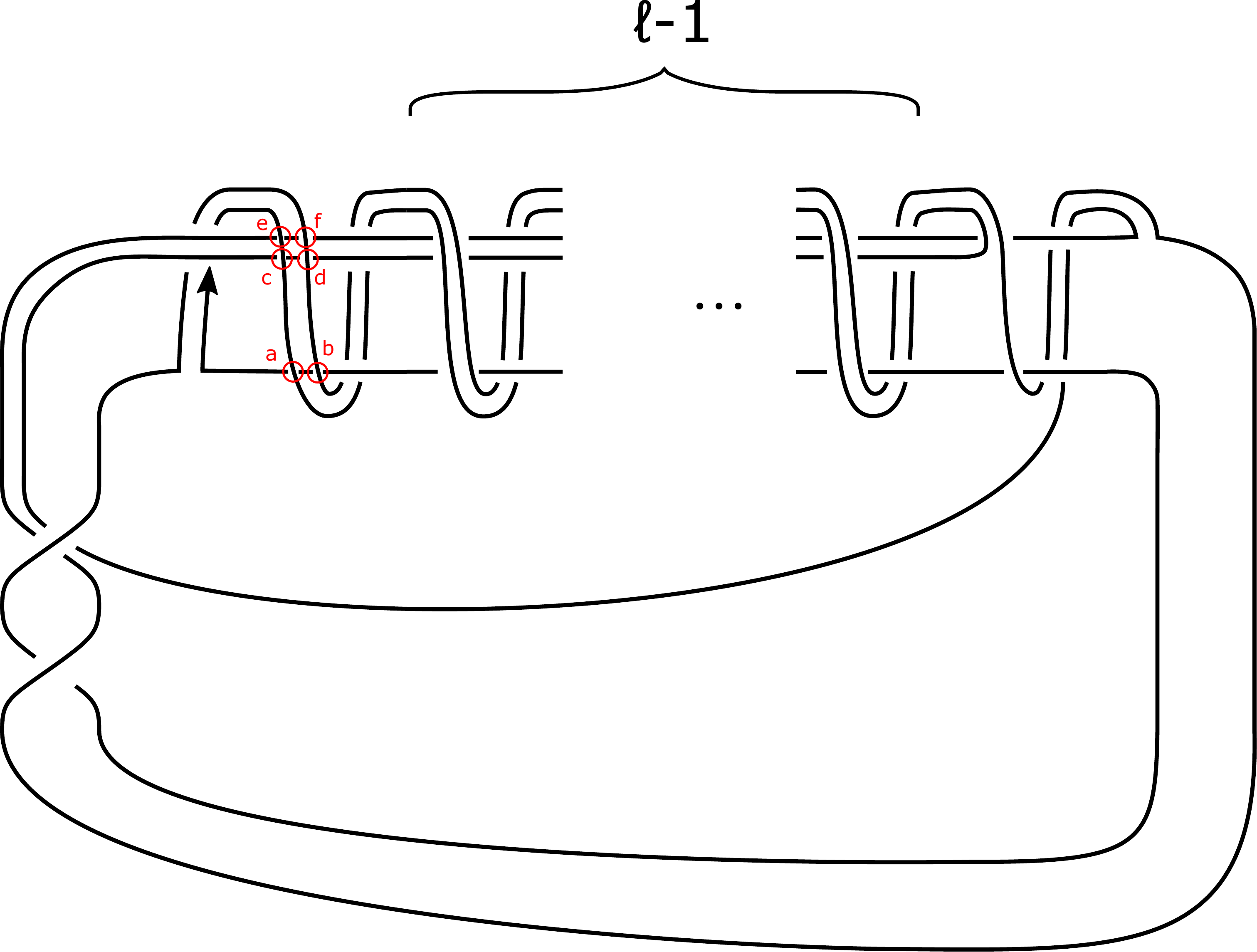}
\caption{The knot $B_{0}^{\ell,0}$. Changing the crossings indicated yields $B_{0}^{\ell-1,0}$.}
\label{fig:B_0l0}
\end{figure}

\begin{figure}
\centering
\begin{subfigure}{.45\textwidth}
\centering
\includegraphics[scale=.15]{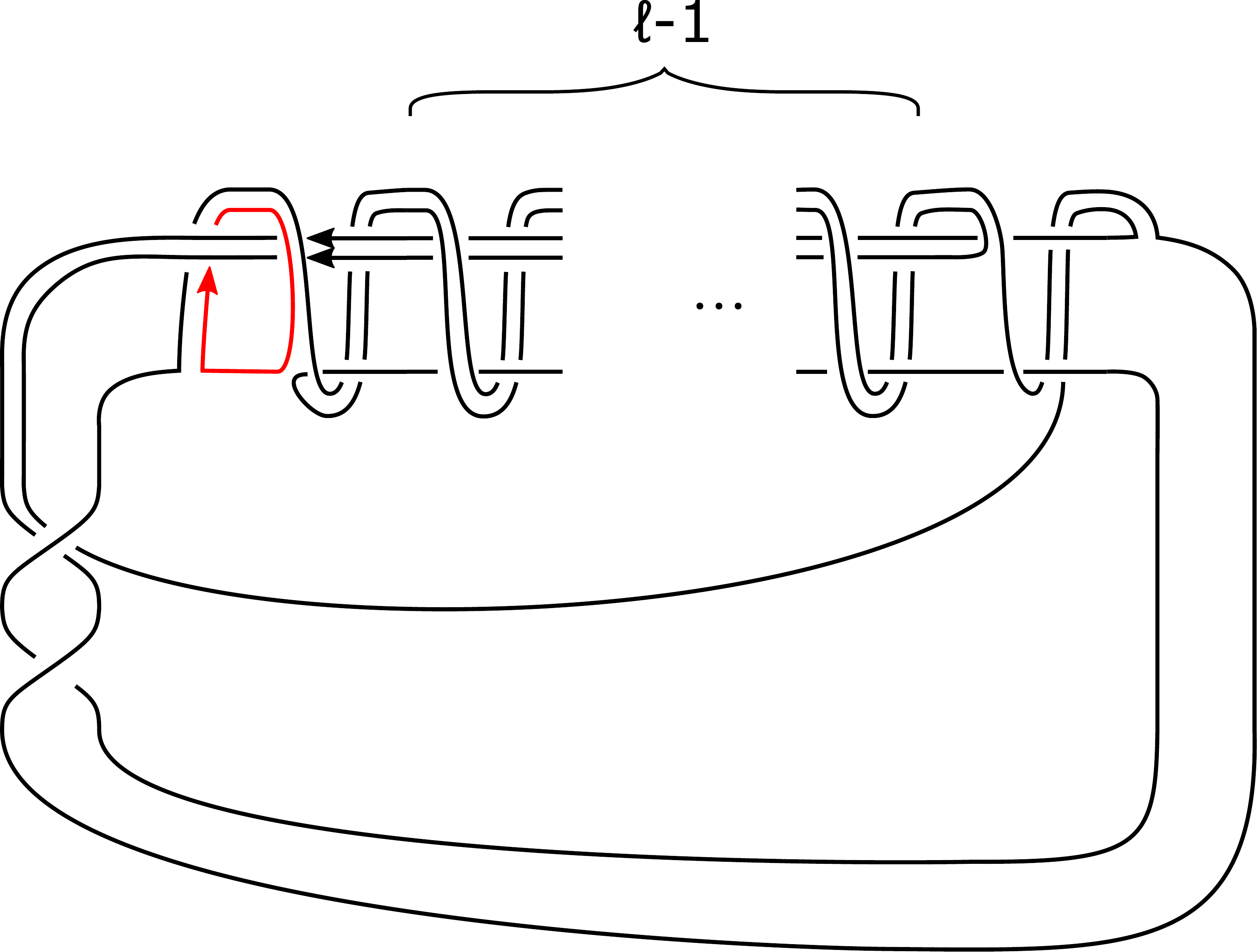}
\caption{linking number $=-2$}
\end{subfigure}
\begin{subfigure}{.45\textwidth}
\centering
\includegraphics[scale=.15]{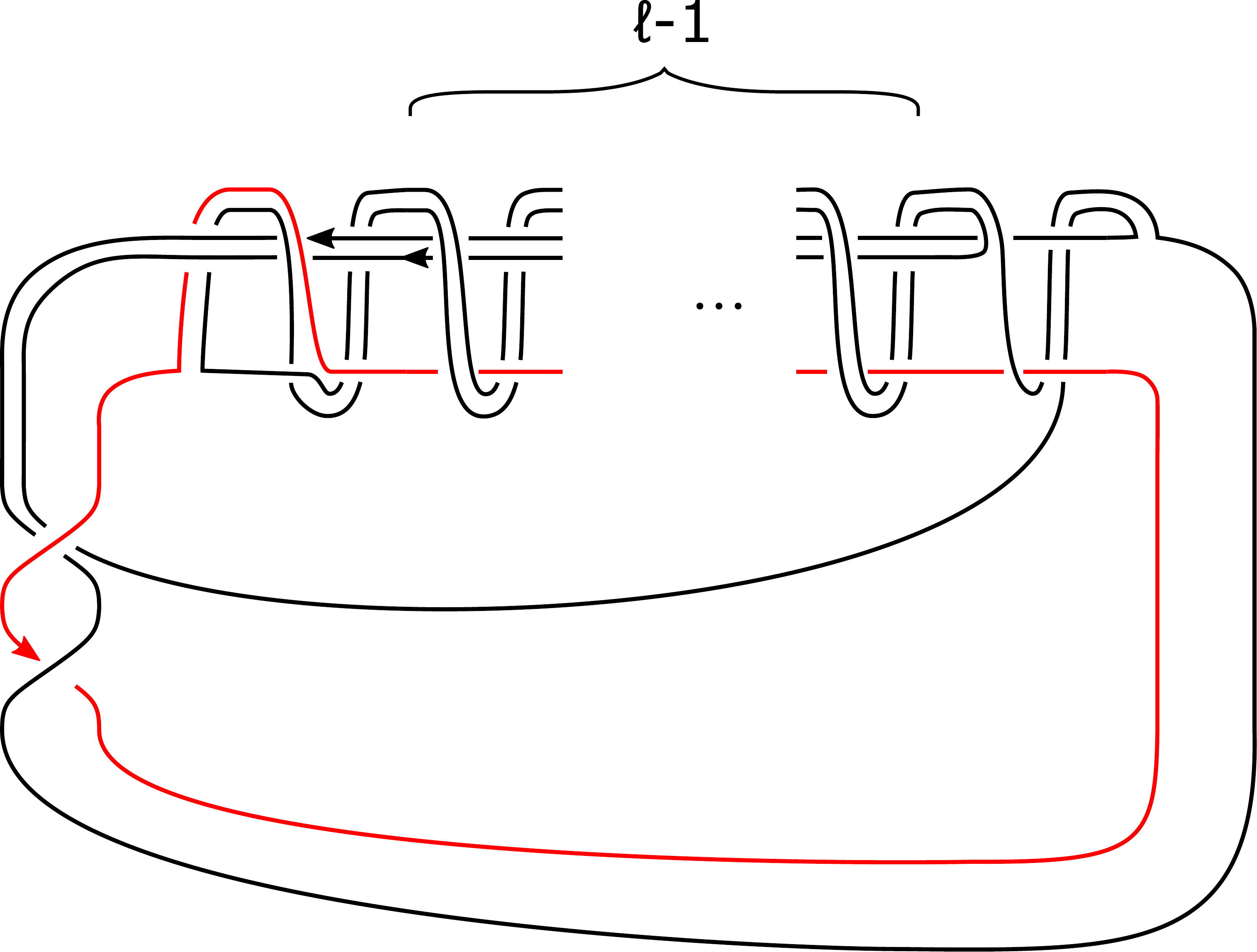}
\caption{linking number $=4$}
\end{subfigure}
\\
\begin{subfigure}{.45\textwidth}
\centering
\includegraphics[scale=.15]{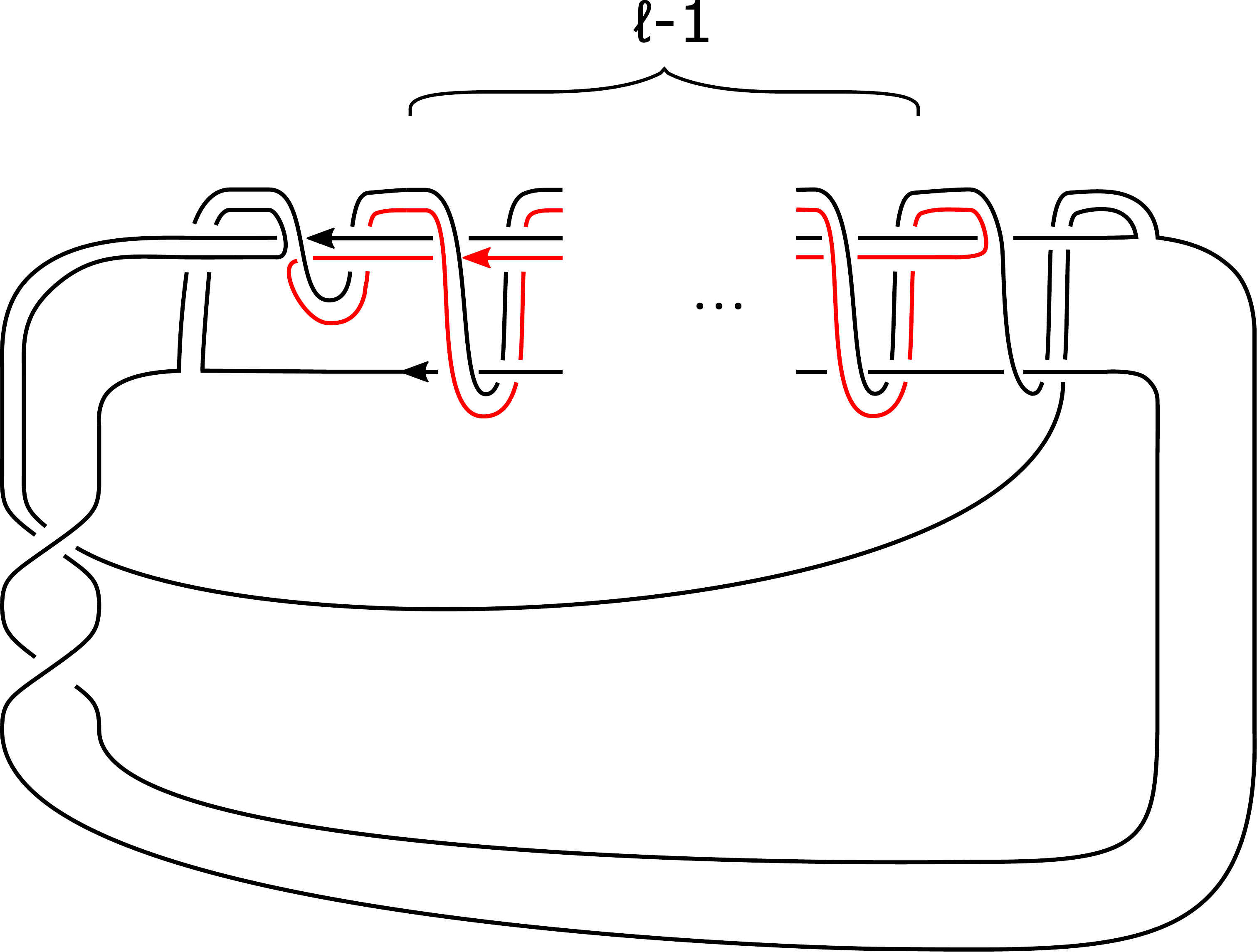}
\caption{linking number $=-\ell-1$}
\end{subfigure}
\begin{subfigure}{.45\textwidth}
\centering
\includegraphics[scale=.15]{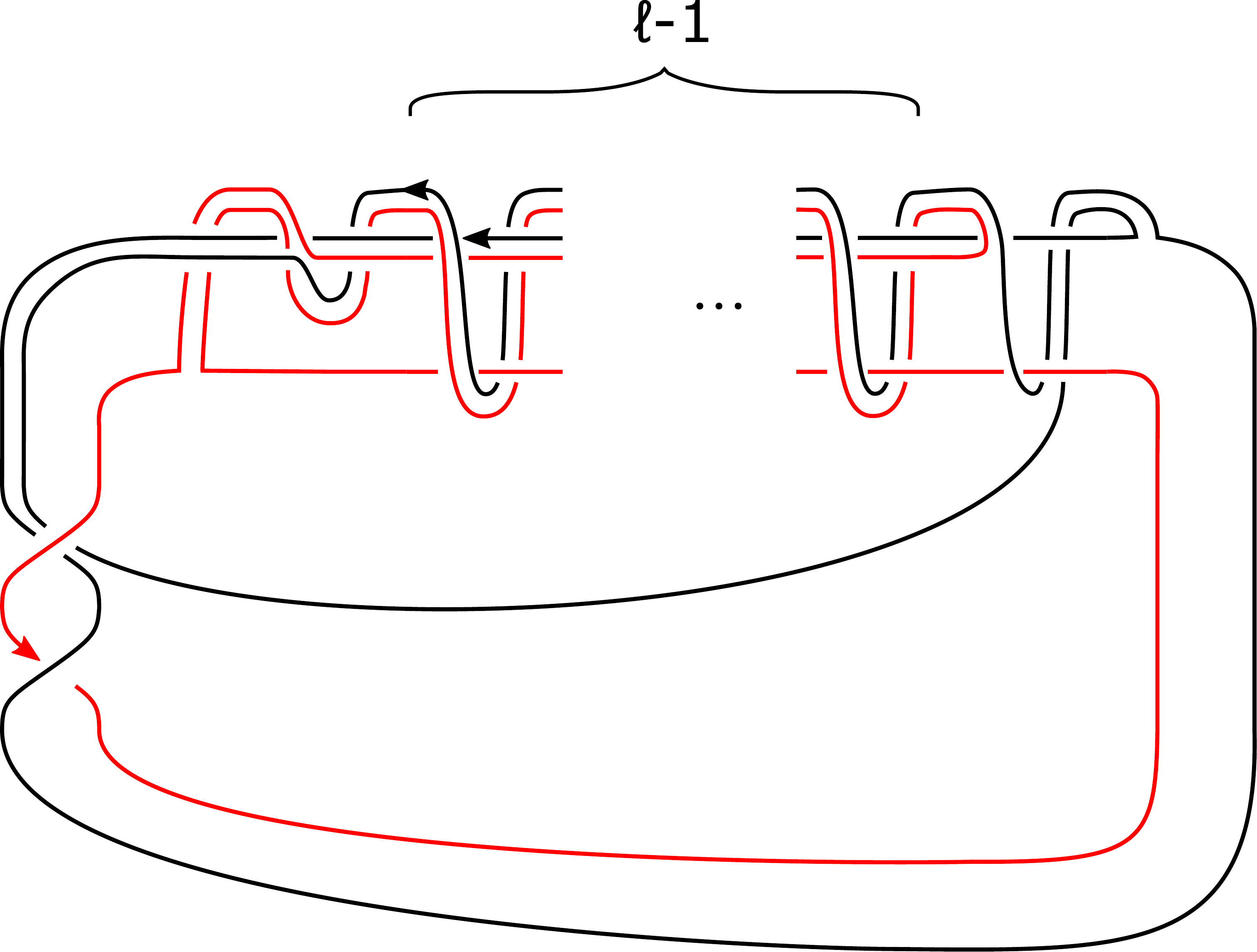}
\caption{linking number $=\ell+1$}
\end{subfigure}
\\
\begin{subfigure}{.45\textwidth}
\centering
\includegraphics[scale=.15]{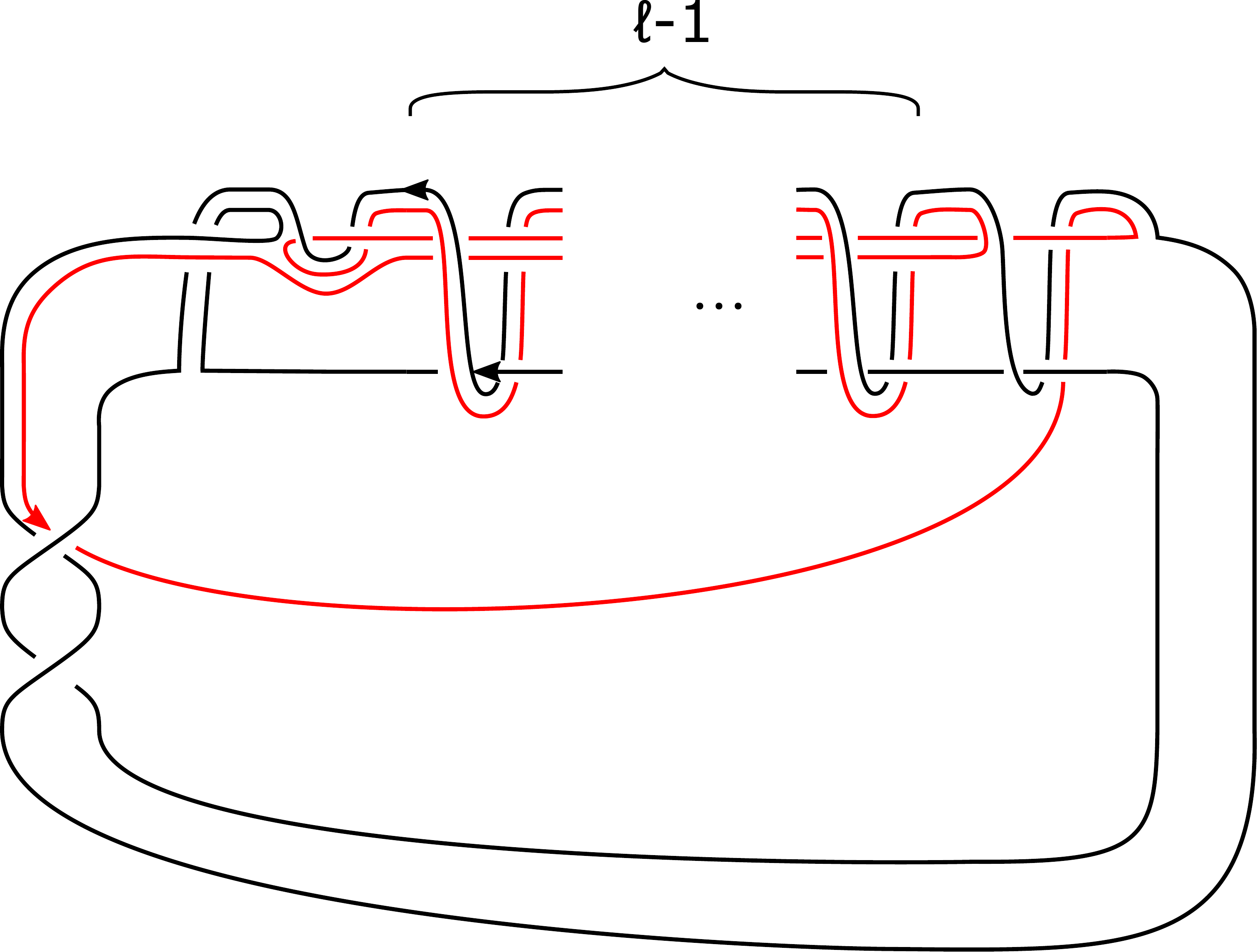}
\caption{linking number $=\ell+1$}
\end{subfigure}
\begin{subfigure}{.45\textwidth}
\centering
\includegraphics[scale=.15]{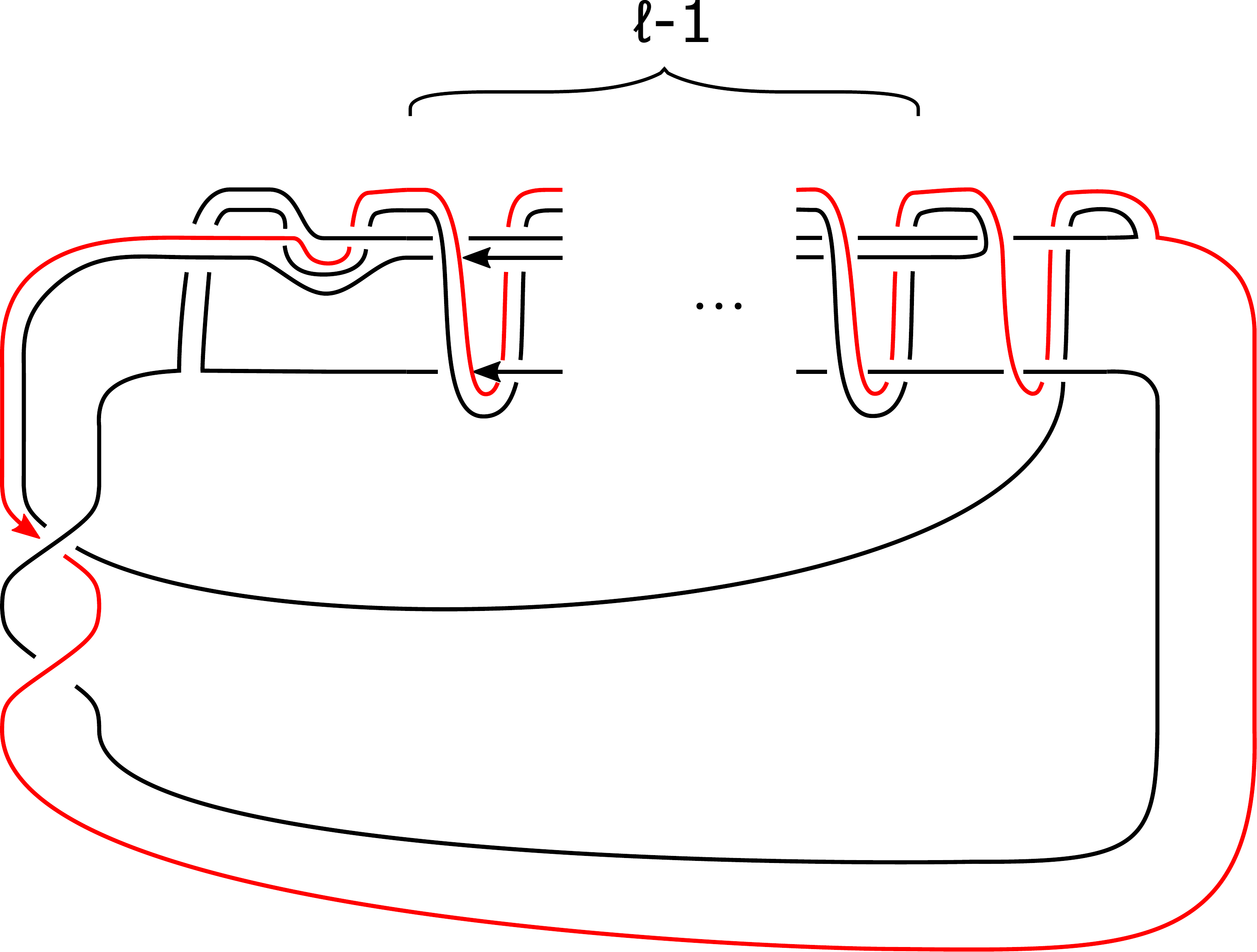}
\caption{linking number $=3\ell+1$}
\end{subfigure}
\caption{Linking numbers of the various links obtained by resolving/changing the crossings in Figure \ref{fig:B_0l0} in sequence.}
\label{fig:cc_B0l0}
\end{figure}
\textbf{Induction on $r$:} We show that $a_2\left(B_0^{\ell,r}\right)= 2\ell^2+r^2+2\ell r + 10\ell +5r+6.$ \\
The base case ($r=0$) is precisely the conclusion of the previous induction on $\ell.$ So we now assume $a_2\left(B_0^{\ell,r-1}\right)= 2\ell^2 +(r-1)^2+2\ell (r-1) + 10\ell +5(r-1)+6$ for some $r\geq 1.$ We see from Figure \ref{fig:B_0lr} that $B_{0}^{\ell,r-1}$ can be derived from $B_{0}^{\ell,r}$ by performing four crossing changes, and Figure \ref{fig:cc_B0lr} records the linking numbers for the links resulting from the oriented resolutions obtained from the indicated sequence of crossing changes (crossing change d is simply a Reidemeister I move).  Once again, repeatedly applying equation \eqref{eq:a2_skein} gives that:
\begin{align*}
a_2\left(B_{0}^{\ell,r}\right) &= a_2\left(B_{0}^{\ell,r-1}\right) - (\ell+1) + (3\ell+2r+4) +1\\
&= a_2\left(B_{0}^{\ell,r-1}\right) +2\ell+ 2r+4\\
&= 2\ell^2 +(r-1)^2+2\ell (r-1) + 10\ell +5(r-1)+6 +2\ell+ 2r+4\\
&=2\ell^2 + r^2 + 2\ell r + 6\ell +5r +6,
\end{align*}
as desired.
\begin{figure}
\centering
\includegraphics[width=.7\textwidth]{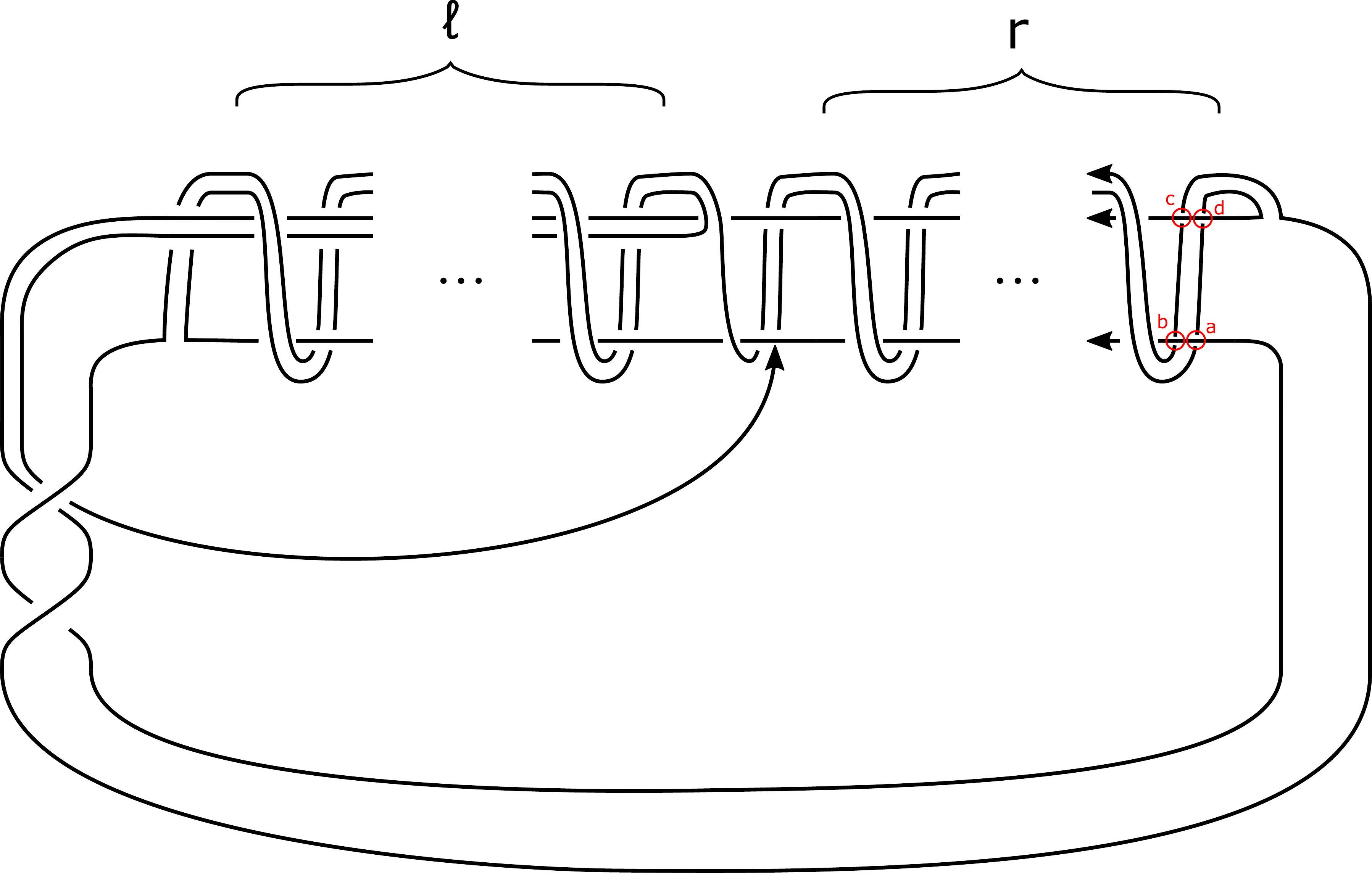}
\caption{The knot $B_{0}^{\ell,r}$. Changing the crossings indicated yields $B_{0}^{\ell,r-1}$.}
\label{fig:B_0lr}
\end{figure}

\begin{figure}
\centering
\begin{subfigure}{.45\textwidth}
\centering
\includegraphics[scale=.15]{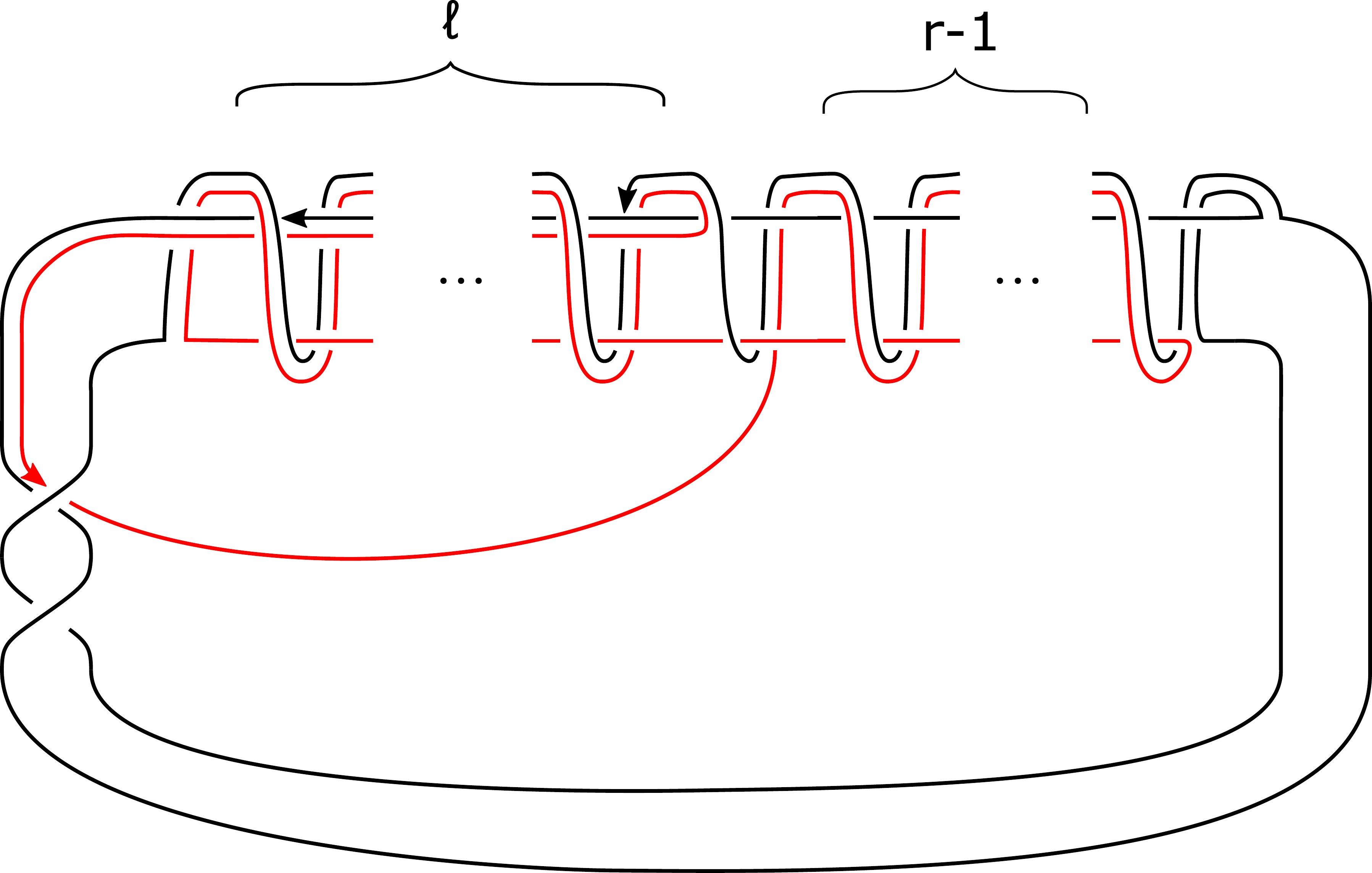}
\caption{linking number $=\ell+1$}
\end{subfigure}
\begin{subfigure}{.45\textwidth}
\centering
\includegraphics[scale=.15]{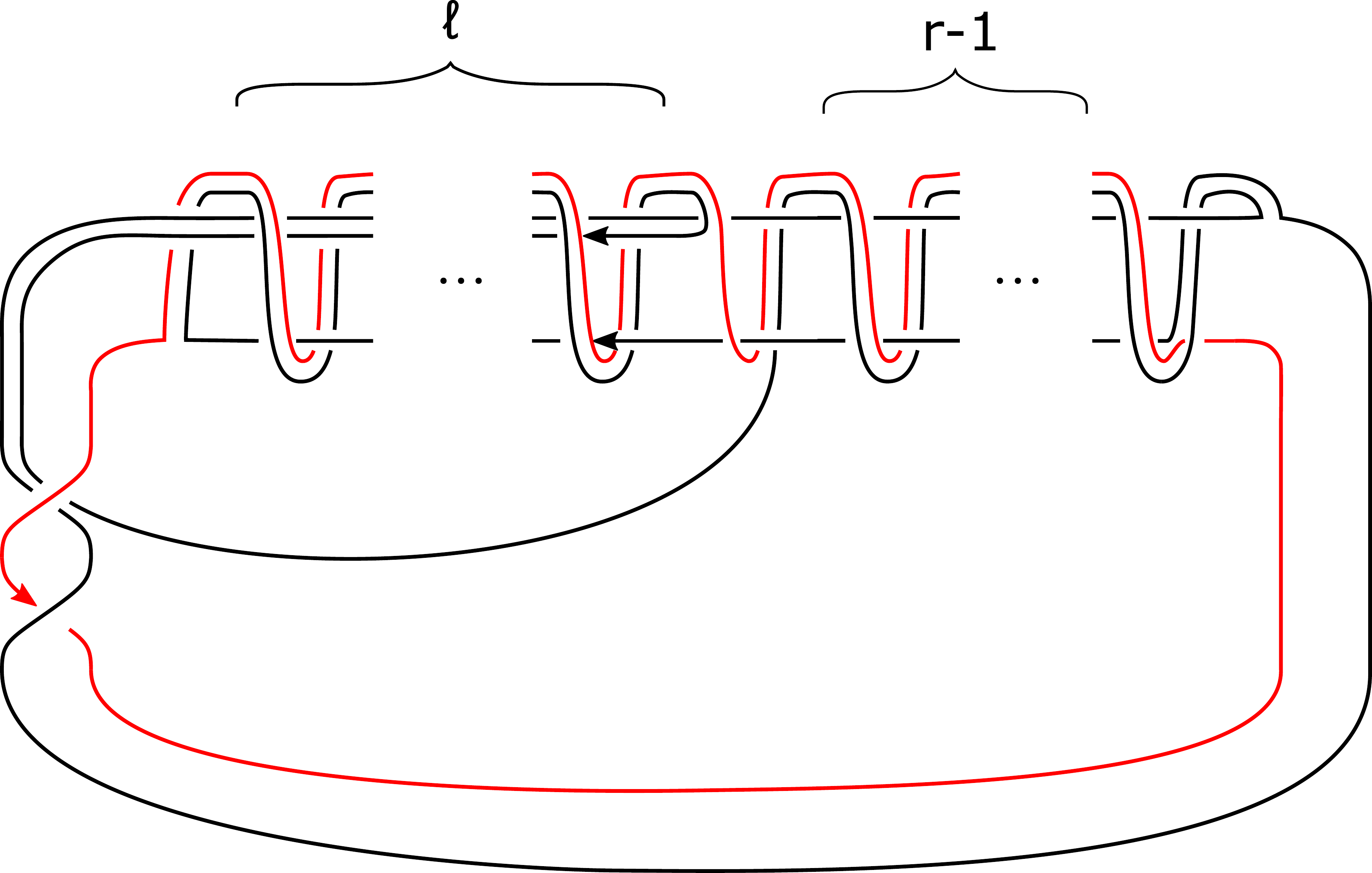}
\caption{linking number $=3\ell+2r+4$}
\end{subfigure}
\\
\begin{subfigure}{.45\textwidth}
\centering
\includegraphics[scale=.15]{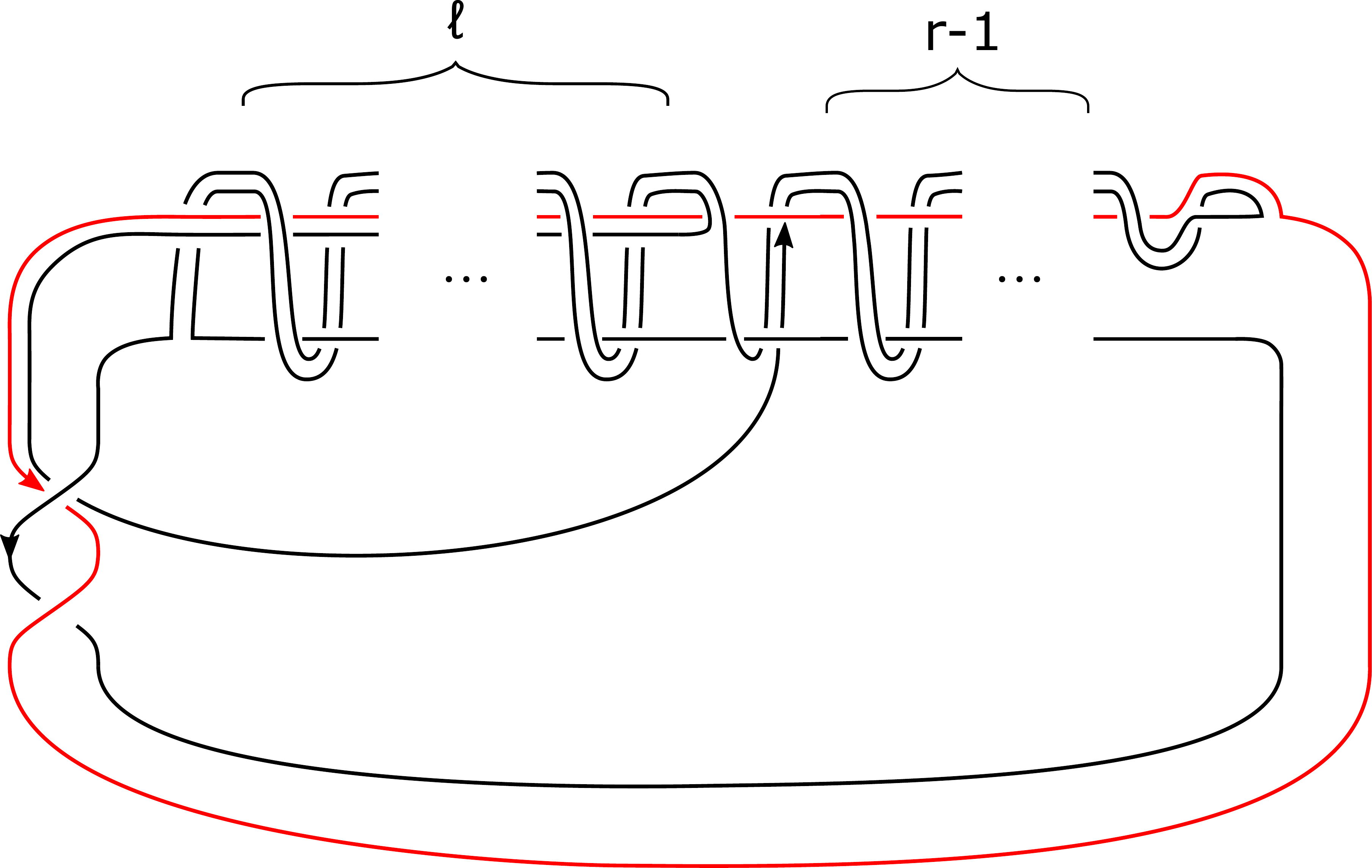}
\caption{linking number $=1$}
\end{subfigure}
\caption{{Linking numbers of the various links obtained by resolving/changing the crossings in Figure \ref{fig:B_0lr} in sequence.}}
\label{fig:cc_B0lr}
\end{figure}
\textbf{Induction on $n$:} We show that $a_2\left(B_0^{\ell,r}\right)= 2\ell^2+r^2+2\ell r + 10\ell +5r-2n+6.$ \\
The base case ($n=0$) is precisely the conclusion of the previous induction on $r.$ So we now assume $a_2\left(B_{n-1}^{\ell,r}\right)= 2\ell^2+r^2+2\ell r + 10\ell +5r-2(n-1)+6$ for some $n\geq 1.$ We see from Figure \ref{fig:B_nlr_ind} that $B_{n-1}^{\ell,r}$ can be derived from $B_{n}^{\ell,r}$ by performing a crossing change, and Figure \ref{fig:cc_Bnlr} records the linking number for the link resulting from the oriented resolution.  Once again, applying equation \eqref{eq:a2_skein} gives that:
\begin{align*}
a_2\left(B_{n}^{\ell,r}\right) &= a_2\left(B_{n-1}^{\ell,r}\right) - 2\\
&=2\ell^2+r^2+2\ell r + 10\ell +5r-2(n-1)+6 -2\\
&=2\ell^2+r^2+2\ell r + 10\ell +5r-2n+6,
\end{align*}
as desired.
\begin{figure}
\centering
\begin{subfigure}{.45\textwidth}
\centering
\includegraphics[scale=.15]{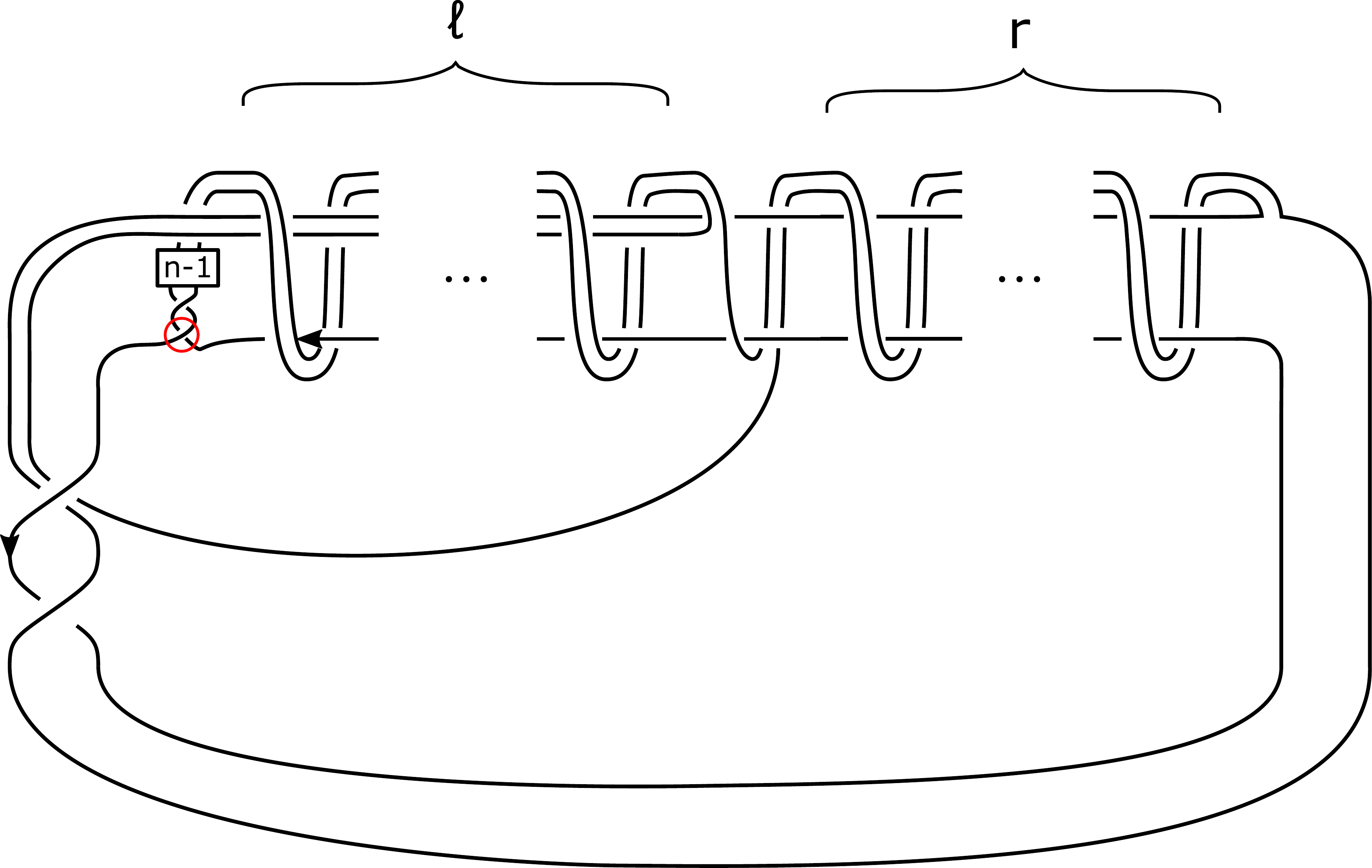}
\caption{The knot $B_{n}^{\ell,r}$. Changing the crossing indicated yields $B_{n-1}^{\ell,r-1}$.}
\label{fig:B_nlr_ind}
\end{subfigure}
\begin{subfigure}{.45\textwidth}
\centering
\includegraphics[scale=.15]{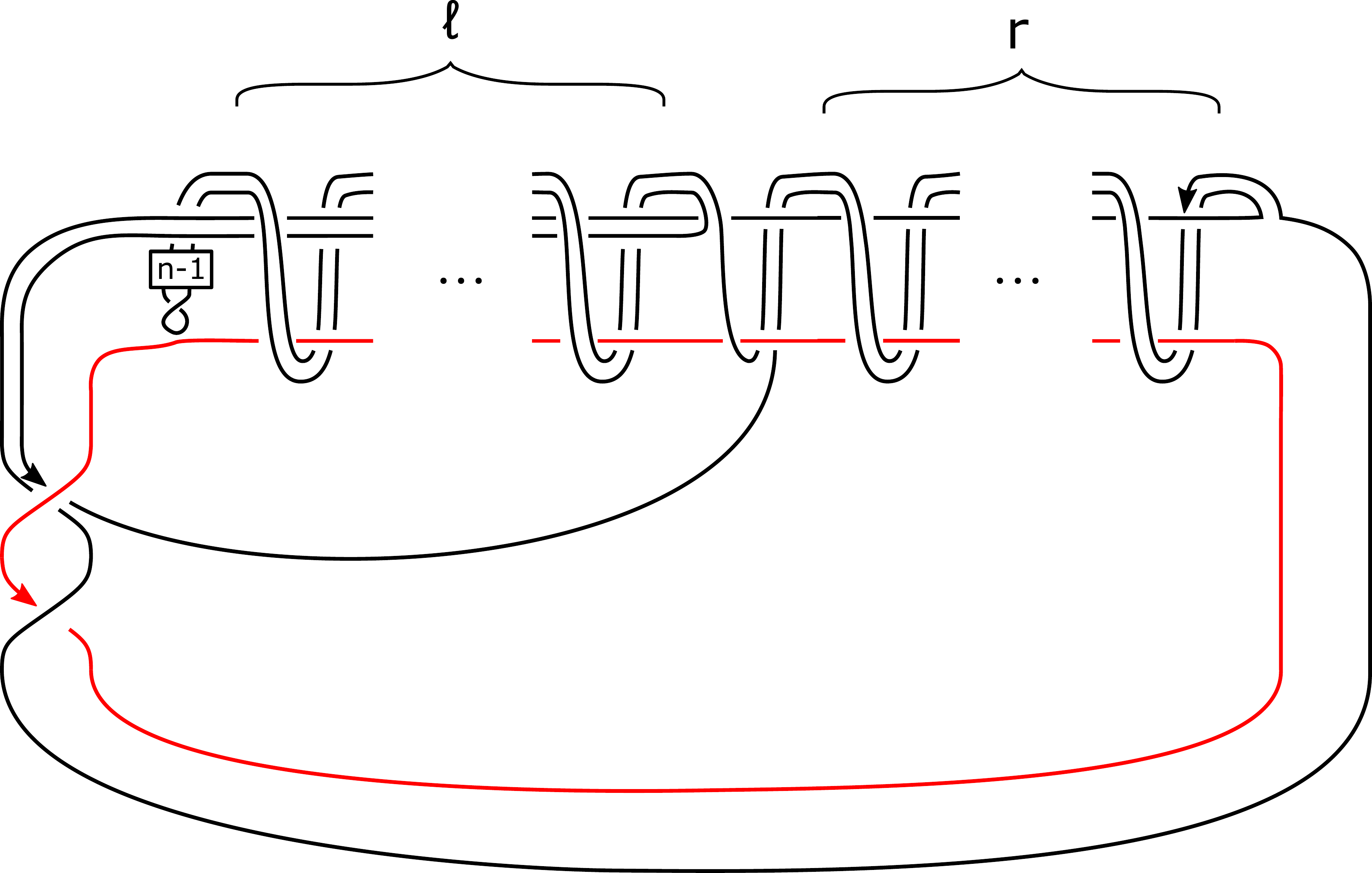}
\caption{linking number $=2$}
\label{fig:cc_Bnlr}
\end{subfigure}
\caption{Linking number obtained by resolving the indicated crossing.}
\end{figure}
\end{proof}

\subsection{Proof of Theorem \ref{thm:main}}

We now prove the main theorem:
\main*

\begin{proof}
The case of $n=1$ was Lemma \ref{lem:K1}. So we suppose $n\geq 2.$ By Theorem \ref{thm:BM} and Lemma \ref{lem:00}, it suffices to show that $\beta_{K_n}$ is not isotopic to a meridian of $K_n.$ If it were, then this isotopy would preserve the linking number of $K_n \cup \beta_{K_n}.$ Hence, by Corollary \ref{cor:a3}, the right-hand-side of \eqref{eq:a3} must equal zero. Applying Lemmas \ref{lem:A} and \ref{lem:B}, we compute:
\begin{align*}
\sum_{k=1}^n &\left( a_2\left(A_n^{n-k,k-1}\right)-a_2\left(B_n^{n-k,k-1}\right)\right)\\
 &=\sum_{k=1}^n \left(4(n-k)^2+(k-1)^2+2(n-k) (k-1) + 6(n-k) +5(k-1)-2n+6\right)\\
 &\qquad -\sum_{k=1}^n  \left(2(n-k)^2+(k-1)^2+2(n-k) (k-1) + 10(n-k) +5(k-1)-2n+6\right)\\
 &=\sum_{k=1}^n \left(2(n-k)^2-4(n-k)\right)\\
 &=\sum_{j=0}^{n-1} (2j^2-4j)\\
 &= \frac{(n-1)(n)(2n-1)}{3}- 2(n)(n-1)\\
 &=n(n-1)\left(\frac{2n-7}{3}\right),
\end{align*}
which is nonzero whenever $n\geq 2.$ Hence, $\beta_{K_n}$ is not isotopic to a meridian of $K_n$ in this case as well.
\end{proof}

\subsection*{Acknowledgements}
The author would like to thank Professor Zolt\'{a}n Szab\'{o}  for  encouraging work on this problem. Thanks also to Sucharit Sarkar for comments on an earlier version of this paper as well as for bringing the author's attention to the family of 2-bridge knots as a potential generalization of these results. The author also thanks Xiliu Yang for pointing out a small error in an earlier draft. This work was supported by the NSF RTG grants DMS-1904628, DMS-1502424, and DMS-1905717.

\bibliographystyle{aomplain}
\bibliography{CharSlpReferences}

\end{document}